\documentclass [11pt,a4paper]{article}
\usepackage[T1]{fontenc}
\usepackage[ansinew]{inputenc}
\usepackage{amssymb}
\usepackage{amsmath}
\usepackage{graphicx}
\usepackage{amsthm}
\usepackage{cite}
\usepackage{color}
\usepackage{mathrsfs}

\newtheorem{theorem}{\textbf{Theorem}}[section]
\newtheorem{lemma}{\textbf{Lemma}}[section]
\newtheorem{proposition}{\textbf{Proposition}}[section]
\newtheorem{remark}{\textbf{Remark}}[section]
\newtheorem{definition}{\textbf{Definition}}[section]
\newtheorem{corollary}{\textbf{Corollary}}[section]

\allowdisplaybreaks[4]

\def\be{\begin{equation}}
\def\ee{\end{equation}}
\def\bes{\begin{equation*}}
\def\ees{\end{equation*}}
\def\bea{\begin{eqnarray}}
\def\eea{\end{eqnarray}}
\def\bt{\begin{theorem}}
\def\et{\end{theorem}}
\def\bl{\begin{lemma}}
\def\el{\end{lemma}}
\def\br{\begin{remark}}
\def\er{\end{remark}}
\def\bc{\begin{corollary}}
\def\ec{\end{corollary}}
\def\bd{\begin{definition}}
\def\ed{\end{definition}}
\def\bp{\begin{proposition}}
  \def\ep{\end{proposition}}

\def\d{\mathbf{d}}
\def\vp{\varphi}

\def\non{\nonumber }

\newtheorem{teorema}{Theorem}[section]
\newtheorem{defin}[teorema]{Definition}
\DeclareMathOperator{\dist}{dist}
\def\Bbb{\mathbb}
\newcommand{\vf}{\varphi}
\newcommand{\bv}{\bar{\mathbf{v}}}
\newcommand{\bvf}{\bar{\varphi}}
\newcommand{\epsi}{\varepsilon}
\newcommand{\bov}{\mathbf{v}}
\newcommand{\T}{\mathbb{T}}

\definecolor{coloras}{rgb}{0.,0.67,0}
\begin{document}
\title{Finite dimensional reduction and convergence to equilibrium for incompressible Smectic-A liquid crystal flows}

\author{ Antonio {\sc Segatti} \thanks{
Dipartimento di Matematica 'F.Casorati', Universit\`{a} di Pavia,
Via Ferrata 1, Pavia 27100, Italy,
\emph{antonio.segatti@unipv.it}.}\ \  and Hao {\sc
Wu}\thanks{Shanghai Key Laboratory for Contemporary Applied
Mathematics and School of Mathematical Sciences,\ Fudan University,
Han Dan Road No. 220, Shanghai 200433,\ P.R. China,
\emph{haowufd@yahoo.com}. Corresponding author.}}

\date{}

\maketitle


\begin{abstract}
\noindent We consider a hydrodynamic system that models the
Smectic-A liquid crystal flow. The model consists of the
Navier-Stokes equation for the fluid velocity coupled with a
fourth-order equation for the layer variable $\vp$, endowed with
periodic boundary conditions. We analyze the long-time behavior of
the solutions within the theory of infinite-dimensional dissipative
dynamical systems. We first prove that in $2D$, the problem
possesses a global attractor $\mathcal{A}$ in certain phase space.
Then we establish the existence of an exponential attractor
$\mathcal{M}$ which entails that the global attractor $\mathcal{A}$
has finite fractal dimension. Moreover, we show that each trajectory
converges to a single equilibrium by means of a suitable
Lojasiewicz--Simon inequality. Corresponding results in $3D$ are
also discussed.

\medskip

\end{abstract}

\section{Introduction}
\setcounter{equation}{0}

Smectic liquid crystal is in a liquid crystalline phase, which
possesses not only some degree of orientational order like the
nematic liquid crystal, but also some degree of positional order
(layer structure). The local orientation of the liquid crystal
molecules is usually denoted by a director field $\d$. In the
nematic state, molecules tend to align themselves along a preferred
direction with no positional order of centers of mass. In the
smectic phase, molecules organize themselves into layers that are
nearly incompressible and of near constant width \cite{de}. The
layers are characterized by the iso-surfaces of a scalar function
$\vp$. A key property that distinguishes the smectic-A liquid
crystals is that, the molecules tend to align themselves along the
direction perpendicular to the layers. The study on the continuum
theory for the smectic-A phase has a long history, see for instance,
\cite{de1,de2,mpp,KP}. A general nonlinear continuum theory for
smectic-A liquid crystals applicable to situations with large
deformations and non-trivial flows was established by E in \cite{E}.
In \cite{E}, the following hydrodynamic system was proposed
 \bea
 &&\rho_t+\bov\cdot\nabla \rho=0,\label{1a}\\
&&\rho \bov_t+\rho \bov\cdot\nabla \bov+\nabla p=\nabla
\cdot(\sigma^e+\sigma^d),\label{2a}\\
&& \nabla\cdot \bov =0,\label{3a} \\
&&\vp_t+\bov\cdot\nabla \vp=\lambda [\nabla \cdot(\xi \nabla
\vp)-K\Delta^2\vp],\label{4a}
  \eea
  where
  \bea
 && \sigma^d=\mu_1(\d^TD(\bov)\d)\d\otimes\d+\mu_4D(\bov)+\mu_5(D(\bov)\d\otimes\d+\d\otimes D(\bov)\d),\non \\
&&\sigma^e =-\xi\d\otimes\d+ K\nabla(\nabla
\cdot\d)\otimes\d-K(\nabla \cdot\d)\nabla \nabla \varphi.\non
  \eea
  In the above system, $\rho$ is the density of the material, $\bov$ is the flow
velocity and $\vp$ denotes the layer variable. In the Smectic-A
phase, molecule orientational direction lies normal to the layer
that $\d=\nabla \vp$. The scalar function $p$ represents the
pressure of the fluid, $\sigma^d$ is the viscous (dissipative)
stress tensor and $\sigma^e$ is the elastic stress tensor (Ericksen
tensor). As usual, $D(\bov)$ indicates the symmetric velocity
gradient, $D(\bov)=\frac12(\nabla \bov+\nabla^\top \bov)$. Due to
the incompressibility of the fluid, there holds $ \nabla \cdot
D(\bov) \,=\, \frac{1}{2}\Delta \bov$. $\mu_1\geq 0$, $\mu_4>0$ and
$\mu_5\geq 0$ are dissipative coefficients in the stress tensor. The
constant $K>0$ arises in the free energy (cf. \cite{E}) and
$\lambda>0$ is elastic relaxation time.

System \eqref{1a}--\eqref{4a} can be viewed as the analog for the
Smectic-A liquid crystal of the Ericksen--Leslie system \cite{eri,
les,de} for the nematic liquid crystal flow. Equation \eqref{1a}
represents the conservation of mass, equation \eqref{2a} is the
conservation of linear momentum, \eqref{3a} implies the
incompressibility of the fluid and equation \eqref{4a} is the
angular momentum equation. $\xi$ is the Lagrange multiplier
corresponding to the constraint associated with the
incompressibility of the layers such that $|\nabla \vp| = 1$. In
order to relax this constraint, an often used approach is to
introduce the Ginzburg-Landau penalization function
$f(\d)=\frac{1}{\epsilon^2}(|\d|^2-1)\d$ ($ 0<\epsilon\leq 1$) with
the associated potential function
$F(\d)=\frac{1}{4\epsilon^2}(|\d|^2-1)^2$ such  that
$f(\d)=\frac{\delta F}{\delta \d}$ (cf. \cite{Liu00, CG10}).
Replacing the original Lagrange multiplier term $\xi\d$ in
$\sigma^e$ as well as in \eqref{4a} by $f(\d)$, we arrive at the
evolution system that will be considered in the present paper:
 \bea
&& \bov_t+ \bov\cdot\nabla \bov-\frac{\mu_4}{2} \Delta \bov+\nabla
p=\nabla
\cdot(\tilde{\sigma}^d+\tilde{\sigma}^e),\label{1b}\\
&& \nabla\cdot v =0,\label{2b} \\
&&\vp_t+\bov\cdot\nabla \vp=\lambda (-K\Delta^2\vp+\nabla\cdot
f(\d)),\label{3b}
  \eea
where
 \bea
 \tilde{\sigma}^d&=&
 \mu_1(\d^TD(\bov)\d)\d\otimes\d+\mu_5(D(\bov)\d\otimes\d+\d\otimes
 D(\bov)\d),\non\\
 \tilde{\sigma}^e&=&-f(\d)\otimes\d+K\nabla(\nabla \cdot\d)\otimes\d-K(\nabla
\cdot\d)\nabla \d.\non
 \eea
 The first well-posedness result of the hydrodynamic system for
  Smectic-A liquid crystal flow mentioned above was obtained in \cite{Liu00}.
 The author considered
 an approximate system like \eqref{1b}--\eqref{3b} but with variable density
 (thus one also has a mass transport equation for $\rho$ like \eqref{1a}) in an open
 bounded domain $\Omega\subset \mathbb{R}^n$, $n=2,3$. The system is
  subject to no-slip boundary condition for $\bov$
  and time-independent Dirichlet--Neumann boundary conditions for $\vp$.
  The author derived the energy dissipative relation of the system
  and proved the existence of global weak solutions in both $2D$
  and $3D$ by using a semi-Galerkin procedure.
  Moreover, he described the global regularity of weak solutions (for
large enough $\mu_4$ if $n = 3$) and provided a preliminary analysis
on the stability of the system. Quite recently, system
\eqref{1b}--\eqref{3b} with constant density and subject to no-slip
boundary condition for $\bov$ but time-dependent Dirichlet--Neumann
boundary data for $\vp$ was studied in \cite{CG10}. The authors
proved the existence of weak solutions that are bounded up to
infinity time for the initial-boundary problem with arbitrary
initial data. The existence of time-periodic weak solutions is also
obtained. Assuming the viscosity $\mu_4$ is sufficiently large, the
author studied the global in time regularity of the solution and
proved the existence and uniqueness of regular solutions for both
the initial-valued problem and the time-periodic problem.

In our present paper, we consider the problem in the $n$-dimensional
torus ($n=2,3$) $\T^n:=\mathbb{R}^n/\mathbb{Z}^n$, namely, system
\eqref{1b}--\eqref{3b} is subjected to periodic boundary conditions.
One of the possible reason for this choice is as follows. Contrary
to the system for nematic liquid crystal flow (cf. e.g.,
\cite{LL95}), now the equation \eqref{3b} for $\vp$ is of fourth
order type and thus lacks of the maximum principle. In particular,
we lose the control of $\|\d\|_{\mathbf{L}^\infty}$. We note that,
the bound of $\|\d\|_{\mathbf{L}^\infty}$ plays an important role in
the subsequent analysis in order to prove the regularity of
solutions to system \eqref{1b}--\eqref{3b} (cf. Lemma \ref{A22d} and
Lemma \ref{hA3}). Higher-order estimates of solutions can be
obtained from some higher-order differential inequalities in the
sprit of \cite{LL95}. However, without the estimate of
$\|\d\|_{\mathbf{L}^\infty}$, we are not able to control certain
higher-order nonlinear terms to derive the required higher-order
differential inequalities. It seems that this is also necessary in
order to complete the calculations in \cite{Liu00}. This difficulty
can be bypassed if one additionally assume that the viscosity
$\mu_4$ is sufficiently large (cf. Lemma \ref{l3d}, see also
\cite{CG10}). In the periodic boundary case, the key observation is
that we can first obtain a uniform estimate on
$\|\d\|_{\mathbf{H}^2}$, which by the embedding
$\mathbf{H}^2\hookrightarrow \mathbf{L}^\infty$ yields the bound of
$\|\d\|_{\mathbf{L}^\infty}$. The proof relies on integration by
parts, thus if we take the boundary conditions as in \cite{Liu00,
CG10}, we are not able to get rid of certain extra boundary terms.

The main propose the present paper is to be a first step towards
 the mathematical study of the long-time behavior of global solutions
 to the periodic boundary problem of system \eqref{1b}--\eqref{3b}.
 In the $2D$ case,
 we are interested in the study of finite dimensional
 global attractors. We recall that a global attractor
 is the smallest compact attracting set of the phase
 space which is fully invariant for the dynamics and
 attracts all the bounded subsets of the phase space
 for large times. Thus, it is certainly a major step
 in the understanding of the long time dynamics of the
 given evolutional system. In particular, when the global attractor
 is proved to have finite fractal or Hausdorff dimension,
 then, although the phase space is infinite dimensional,
 the dynamics of the system becomes finite
 dimensional for large times and can be described with a finite numbers of
 parameters. This is the so called finite dimensional reduction.
 We refer to \cite{TE} for a detailed description.
We will prove the finite dimensionality of the global attractor by
showing the existence of an exponential attractor, which is a
semi-invariant, compact set attracting exponentially fast the
bounded subsets of the phase space. Moreover, it has finite fractal
dimension and contains the global attractor. We refer to \cite{EFNT}
and to \cite{MZsur} for a detailed introduction of this concept and
for discussion on its importance. This approach has the advantage
that, contrary to the volume contraction method (see \cite{TE}), it
does not need any differentiability property of the semigroup.
 As a second step, we will study the long-time behavior of single trajectories, i.e., the convergence
 to single equilibrium. This is a nontrivial problem because the structure of the set of equilibria can be
quite complicated and, moreover,  may form a continuum. In
particular, under our current periodic boundary conditions, one may
expect that the dimension of the set of equilibria is at least $n$.
This is because a shift in each
 variable should give another steady state. Moreover, we note that for our
 system, every constant vector $\d_0$ with unit-length ($|\d_0|=1$) serves as an absolute
 minimizer of the functional $E$ in \eqref{Evp}. We shall
apply the \L ojasiewicz--Simon approach (cf. L. Simon \cite{S83}) to
prove the convergence and obtain estimates on the convergence rate
(see \cite{Chi,J981,HJ01,WGZ1,RH98,W10,GG10} and the references
therein for applications to various evolution equations). In $3D$
case, some partial results can be obtained. Since the
$\mathbf{L}^\infty$-estimate of $\d$ is still available, we can show
the local existence of strong solutions for arbitrary initial data
by higher-order energy estimates. Assuming the viscosity $\mu_4$ is
sufficiently large, we also obtain the global existence of strong
solution. Finally, we show that the global weak/strong solutions
will converge to single equilibrium as in the $2D$ case. In
particular, we prove the well-posedness and long-time behavior of
global strong solutions when the initial data is close to a local
minimizer of the energy $E$ using the \L ojasiewicz--Simon
inequality, which improves the results in the literature that only
the case near an absolute minimizer is considered (cf. \cite{Liu00},
see also \cite{LL95,W10} for the nematic liquid crystal flow).

The remaining part of the paper is organized as follows. Section 2
is devoted to some preliminaries and the main results of the paper.
In Section 3, we prove that in the $2D$ case, the semigroup
generated by our model on a suitable phase space possesses the
global attractor $\mathcal{A}$ and an exponential attractor
$\mathcal{M}$. This allows us to infer that $\mathcal{A}$ has finite
fractal dimension. In Section 4, in the $2D$ case, we demonstrate
that each trajectory converges to a single equilibrium and also find
a convergence rate estimate. Finally, in Section 5, we discuss the
results in $3D$ case.


\section{Preliminaries and Main Results}
\setcounter{equation}{0}

We denote the Lebesgue spaces with $L^p(\mathbb{T}^n)$ (or simply $L^p$), $p\in [1,
\infty]$, and their norms with $\|\cdot\|_{L^p}$. When $p = 2$, we simply
denote the $L^2$-norm by $\|\cdot\|$ and its inner product by
$(\cdot, \cdot)$. With $H^s$, $s\,\in\,\mathbb{R}$ we indicate the
Sobolev spaces $H^s(\mathbb{T}^n)$ endowed with norm
$\|\cdot\|_{H^s}$. To simplify the notations, we will denote the
vector spaces $(L^p)^n$, $(H^s)^n$, $(L^p)^{n\times n}$,
$(H^s)^{n\times n}$... by $\mathbf{L}^p$ and $\mathbf{H}^s$,
respectively, and their norms are denoted in the same way as above.
For any norm space $X$, we denote its subspace by $\dot X$
 such that $\dot X=\{w\in X: \int_{\T^n} w dx=0\}$.
As customary, we introduce the following standard functional spaces
for the Navier--Stokes equation
 \be  H:=\{\bov\in \mathbf{L}^2(\mathbb{T}^n),\ \nabla\cdot
 \bov=0\},\quad
 V:=\{\bov\in \mathbf{H}^1(\mathbb{T}^n),\ \nabla\cdot
 \bov=0\},\quad V':=\text{the\ dual of\ } V.\non
 \ee
 $\langle\cdot,\cdot\rangle$ denotes the duality product between $V'$ and
 $V$. The shorthand
notation $D_{ij}$ will be used for the entries of the matrix $D$. We
indicate with the same symbol $C$ different constants. Special
dependence will be indicated if it is necessary. Analogously,
$\mathcal{D}\,:\,\mathbb{R}^+\to \mathbb{R}^+$ denotes a generic
monotone function. Throughout the paper, the Einstein summation
convention will be used.
\\

We introduce the notions of  weak/strong solutions to problem
\eqref{1b}--\eqref{3b}:
 \bd
 (1) $(\bov,\vf)$ is a weak
solution to problem \eqref{1b}--\eqref{3b} in
$[0,T)\,\times\,\mathbb{T}^n$ {\rm ($T\in (0,+\infty)$)}, if
$\bov\,\in\,L^\infty(0,T;H)\cap L^2(0,T;V)$, $\vf\in
L^\infty(0,T;H^2)\cap L^2(0,T;H^4)$ and verifying
 \bea
&& \langle \partial_t \bov, \mathbf{w}\rangle + ((\bov\cdot\nabla)
\bov,\mathbf{w}) + \frac{\mu_4}{2}(\nabla \bov,\nabla \mathbf{w})
 = (\tilde{\sigma}^d+\tilde{\sigma}^e,\nabla \mathbf{w}),\,\,\,\,\forall \mathbf{w}\,\in V,\non\\
&&\vp_t+\bov\cdot\nabla \vp=\lambda (-K\Delta^2\vp+\nabla\cdot
f(\nabla \vp)),
 \,\,\,\,\text{ a.e. in } [0,T]\times \T^n \non\\
 &&  \bov(0)\,=\,\bov_0,\,\,\,\,\,\,\,\vf(0)\,=\,\vf_0, \quad  \text{in} \ \mathbb{T}^n.\non
 \eea

 (2) A weak solution $(\bov,\vf)$ to problem
\eqref{1b}--\eqref{3b} is a strong solution, if for $T\,>\,0$,
$\bov\,\in\,L^\infty(0,T;V)\cap L^2(0,T;\mathbf{H}^2)$, $\vf\in
L^\infty(0,T;H^4)\cap L^2(0,T;H^6)$, and the system
\eqref{1b}--\eqref{3b} is satisfied point-wisely in $[0,T)\times
\T^n$.
 \ed

The calculation in \cite{Liu00} (with different boundary conditions
but the proof is the same) implies that system
\eqref{1b}--\eqref{3b} has a dissipative nature, in particular,
 the following \emph{basic energy law} holds
\begin{proposition} \label{energy1}
Let $(\bov,\vp)$ be a smooth solution to the system
\eqref{1b}--\eqref{3b}. Define the total energy
 \be
\mathcal{E}(t)=\frac{1}{2}\|\bov(t)\|^2+\frac{K}{2}\|\Delta
\varphi(t)\|^2+\int_{\T^n}F(\d)(t)\,dx.
 \label{E}
 \ee
 Then following identity holds:
 \bea  \frac{d}{dt}\mathcal{E}(t)&=&
-\int_{\T^n}\left(\mu_1(\d^\top D(\bov)\d)^2+\frac{\mu_4}{2}|\nabla
\bov|^2+2\mu_5|D(\bov)\d|^2 \right)dx \non\\
&& -\lambda \left\|-K\Delta^2
 \varphi+\nabla\cdot f(\d)\right\|^2. \label{energy}
 \eea
\end{proposition}

We can prove the existence of weak solutions to
\eqref{1b}--\eqref{3b} by applying the semi-Galerkin approximation
scheme as in \cite{Liu00} (cf. also \cite{Lions,LL95,CG10}). The
proof is similar to \cite{Liu00, CG10} and we omit the details here.

 \bt \label{existence} {\rm [Existence of weak solution]} Suppose $n=2,3$. For any $(\bov_0,\vf_0)\,\in\,H\times H^2$,
system \eqref{1b}--\eqref{3b} admits at least one weak solution.
 \et

A weak/strong uniqueness result was obtained in \cite{CG10} for
system \eqref{1b}--\eqref{3b} with different boundary conditions
(see \cite{Liu00} for a statement for the system with variable
density). A similar argument yields the same conclusion for our
case:
 \bt \label{uniqueness} {\rm [Weak/strong uniqueness]} If $(\bov_1,\vf_1)$ and $(\bov_2,\vf_2)$ are
respectively a weak and a strong solution of \eqref{1b}--\eqref{3b}
in $[0,T]$, then $(\bov_1,\vf_1)\,\equiv\,(\bov_2,\vf_2)$ almost
everywhere in $[0,T]\times \mathbb{T}^n$.
 \et

Here are the main results of the paper:

\bt \label{reg}
 Suppose $n\,=\,2$.

 (1) Any weak solution to system \eqref{1b}--\eqref{3b} becomes strong for
strictly positive times such that for any $t>0$,
\begin{equation}
\label{strong-estimate} \|(\bov,\vf)(t)\|_{V\times H^4} +
\int_{t}^{t+1}\|\Delta \bov(s)\|^2 + \|\vf(s)\|^2_{H^6}ds \,\le\,
\mathcal{D}(\|(\bov_0,\vf_0)\|_{H\times H^2}, t),
\end{equation}
$\mathcal{D}$ is a positive function depending on
 $\|\bov_0\|$, $\|\vp_0\|_{H^2}$, $t$ and coefficients of the system.
 In particular, $\lim_{t\to 0^+}\mathcal{D}(t)=+\infty$.

(2) For any $(\bov_0,\vf_0)\,\in\,V\times H^4$, system
\eqref{1b}--\eqref{3b} admits a unique strong solution.
 \et

 \bt \label{2datt} Suppose $n=2$. Denote the
phase space $\mathcal{H} \times H^2_c$, where $\mathcal{H}=\{\bov\in
H: \int_{\T^2} \bov dx= \mathbf{h}\}$ and $H^2_c=\{\vp\in H^2:
\int_{\T^2} \vp dx = c\}$, with $\mathbf{h}$ being any given
constant vector in $\mathbb{R}^2$ and $c$ is an arbitrary constant.

  (1) System \eqref{1b}--\eqref{3b} processes a
global attractor $\mathcal{A}$ with finite fractal dimension in
$\mathcal{H} \times H^2_c$. Moreover, $\mathcal{A}$ is bounded in
$V\,\times\,H^4$ and it is generated by all the complete
trajectories.

  (2) System \eqref{1b}--\eqref{3b} possesses an exponential attractor
$\mathcal{M}$ in  $\mathcal{H} \times H^2_c$, which is bounded in $V
\times H^4$.
 \et

\bt
 \label{con2d}
 Suppose $n=2$. For any $\bov_0\in \dot H$, $\vp_0\in H^2$, the
  global weak solutions to problem \eqref{1b}--\eqref{3b} has the
following property:
 \be \lim_{t\rightarrow +\infty}
 (\|\bov(t)\|_{\mathbf{H}^1}+\|\vp(t)-\vp_\infty\|_{H^4})=0,\label{cgce}
 \ee
 where $\varphi_\infty\in H^4$ is a solution to the following
 periodic elliptic problem:
  \be
  - K \Delta^2\vp_\infty + \nabla \cdot f(\nabla \vp_\infty)=0,\quad x\in
  \T^2, \text{with}\   \int_{\T^2}\varphi_\infty dx=\int_{\T^2} \varphi_0 dx,
   \label{staa}
   \ee
 Moreover, there exists a positive constant $C$ depending on
 $\bov_0,\vp_0, \vp_\infty, K, \lambda, \mu's$
 such that
 \be
 \|\bov(t)\|_{\mathbf{H}^1}+\|\vp(t)-\vp_\infty\|_{H^4}\leq C(1+t)^{-\frac{\theta}{(1-2\theta)}}, \quad \forall\ t \geq
 1.\label{rate}
 \ee
$\theta \in (0,\frac12)$ is usually called \L ojasiewicz exponent
and it depends on $\vp_\infty$.

 For any $\bov_0\in \dot V$, $\vp_0\in H^4$, the
  global strong solution to problem \eqref{1b}--\eqref{3b} has the
  same property \eqref{cgce} and \eqref{rate} holds for $t\geq 0$.
 \et

 \bt \label{3d} Suppose $n=3$.

 (1) For any $(\mathbf{v}_0, \vp_0)\in V\times H^4$,
problem \eqref{1b}--\eqref{3b} admits a unique local strong
solution.

 (2) For any $(\mathbf{v}_0, \vp_0)\in V\times H^4$, if $\mu_4\geq
 \underline{\mu}_4(\mathbf{v}_0,\vp_0)$ is sufficiently large (cf. \eqref{mu4b}), problem \eqref{1b}--\eqref{3b}
  admits a unique global strong solution.

 (3) Let $(\mathbf{v},\vp)$ be the weak solution to problem \eqref{1b}--\eqref{3b} on $[0,+\infty)$.
 Then there is some $T^*>0$ such that $\mathbf{v}\in L^\infty(T^*, \infty; V)\cap L^2_{loc}(T^*,\infty; \mathbf{H}^2)$, $\vp\in
 L^\infty(T^*,\infty; H^4)\cap L^2_{loc}(T^*,\infty; H^6)$.

 (4) Let $\vp^*\in H^2$ be a local/absolute  minimizer of
 $E(\vp)$ (cf. \eqref{Evp}). For any $\bov_0\in V$, $\vp_0\in H^4$ satisfying
  $\|\bov_0\|_{\mathbf{H}^1}\leq 1$,  $\|\vp_0-\vp^*\|_{H^4}\leq
 1$, there are constants $\sigma_1,\sigma_2\in (0,1]$ which depend
 on $\vp^*$  and coefficients of the system such that if
 $\|\bov_0\|\leq \sigma_1$ and  $\|\vp_0-\vp^*\|_{H^2}\leq \sigma_2$,
then problem
 \eqref{1b}--\eqref{3b} admits a unique global strong solution.

 (5) If we further assume that $\int_{\T^3} \bov_0 dx=0$, then the global weak/strong
 solution to \eqref{1b}--\eqref{3b} enjoys the same long-time behavior as in
 Theorem \ref{con2d}, with $t\geq 1$ in \eqref{rate} being replaced
 by $t\geq T^*$ for the weak solution.
 \et

\br\label{zero}
  Due to the periodic boundary conditions, we can
easily see that the mean value of $\bov$ and $\vf$ are conserved in
the evolution:
$$ \int_{\T^n} \bov(t) dx=\int_{\T^n} \bov_0 dx, \quad
\int_{\T^n}\vf(t) dx=\int_{\T^n}\vf_0 dx, \quad \forall\ t\geq 0.$$
For the sake of simplicity, by replacing $\bov$ (respectively $\vf$)
with $ \bov_0 - \int_{\T^n} \bov_0 dx$ (respectively with $\vf_0
-\int_{\T^n}\vf_0 dx$), we shall always assume that
$\int_{\T^n}\bov_0 dx\,\equiv\,0$ and $\int_{\T^n} \vf_0 dx
\,\equiv\, 0$ in the subsequent proof. Since system
\eqref{1b}--\eqref{3b} is invariant under a shift of $\vp$ by any
constant, the transformation on $\vp$ will not influence all our
results. However, when we shift the velocity $\bov$ to make it has a
zero mean, there will be one extra lower-order term in the equations
\eqref{1b} and \eqref{3b} respectively. This difference will not
influence most results we obtain except the convergence of global
solutions to equilibria (Theorem \ref{con2d} and point (5) in
Theorem \ref{3d}). If the mean value of $\bov$ is not zero, we
cannot apply the Poincar\'e inequality to obtain the decay of
$\|\bov\|_{\mathbf{H}^1}$ from the convergence of $\|\nabla \bov\|$.
 \er

 \br If we simply set $\bov=0$, system \eqref{1b}--\eqref{3b} is reduce to
the single equation $\vp_t=\lambda (-K\Delta^2\vp+\nabla\cdot
f(\nabla \vp))$, which has been used to model epitaxial growth of
thin films with slope selection in $2D$, where $\vp$ denotes a
scaled height function of a thin film (cf. \cite{LJG03, KY}).
Existence and uniqueness of the weak solutions as well as some
preliminary results on long-time behavior of the solutions as time
goes to infinity (like sequent convergence) was obtained in
\cite{LLW04}.
 \er


\section{Global Attractor and Exponential Attractors in $2D$}
\setcounter{equation}{0}

In this section we study the long time behavior of the system
\eqref{1b}--\eqref{3b} in terms of global and exponential
attractors. As suggested by Remark \ref{zero}, we work in the phase
spaces
\begin{equation*}
\Phi\,:=\dot H\times \dot H^2,\quad \Phi_{1}\,:=\, \dot V\times \dot
H^4
\end{equation*}
with the norms $\|(\bov,\vf)\|_{\Phi}^2\,:=\,\|\bov\|^2_{H} +
\|\vf\|^2_{H^2}$, $\|(\bov,\vf)\|_{\Phi_1}^2\,:=\,\|\bov\|^2_{V} +
\|\vf\|^2_{H^4}$, respectively. It is obvious that $\Phi_{1}$ is
compactly embedded into $\Phi$.

 Recall the definition of the global attractor (cf.  \cite{TE})
 \bd Suppose $\mathcal{X}$ is a complete metric space. Given a semigroup $S(t):\mathcal{X}\mapsto\mathcal{X}$, a subset
$\mathcal{A}\subset\mathcal{X}$ is the global attractor if
 (i) The set $\mathcal{A}$ is compact in $\mathcal{X}$; (ii) It is strictly invariant: $S(t)\mathcal{A}=\mathcal{A}$, $t\ge 0$;
 (iii) For every bounded set $B\,\subset \,\mathcal{X}$ and for every neighborhood $\mathcal{O}=\mathcal{O}(\mathcal{A})$ of
 $\mathcal{A}$ in $\mathcal{M}$, there exists a time $T=T(\mathcal{O})$ such that
 $ S(t)B\subset \mathcal{O}(\mathcal{A})$ for all $t\ge T$.
 \ed

 As far as our system is concerned, we define $S(t): \Phi\mapsto\Phi$ to be the map $(\bov_0,\vf_0)\mapsto
 (\bov(t),\vf(t))$.
 Unfortunately, Theorem \ref{existence} does not guarantee that $S(t)$ is well defined on the phase space
 $\Phi$, since we are not able to prove a uniqueness result for weak solutions. We will refer to $S(t)$ as a
 \emph{solution operator}, being aware of the fact that, in principle, $S(t)(\bov_0,\vf_0)$ could be
 multi-valued due to the possible non-uniqueness. In the cases in which uniqueness holds,
 with a little abuse of notation, we will still indicate with $S(t)$ the corresponding \emph{semigroup}.
 As a consequence of the possible non-uniqueness, as it will be further explained later, we will
 not directly construct the global attractor on the phase space $\Phi$ but rather on the "lifted" phase space of $\ell$-trajectories.

\subsection{Dissipativity}
 The following lower-order uniform estimate follows from the basic energy
 law:
\begin{lemma}
\label{low} Suppose $n=2,3$. For $\bov_0\in \dot H$, $\varphi_0\in
\dot H^2$, the weak solution to \eqref{1b}--\eqref{3b} has the
following uniform estimates
 \be
 \|\bov(t)\|+\|\varphi(t)\|_{H^2}\leq C, \quad t\geq 0, \label{unilow}
 \ee
 where $C>0$ is a constant depending on $\|\bov_0\|,
 \|\varphi_0\|_{H^2}, K$. Moreover,
 \be \int_0^{+\infty} \left(\|\nabla \bov(t)\|^2+\left\|-K\Delta^2
 \varphi(t)+\nabla\cdot f(\d(t))\right\|^2\right)dt\leq
 \max\left\{\frac{2}{\mu_4}, \frac{1}{\lambda}\right\}\mathcal{E}(0).\label{int}
 \ee
\end{lemma}
Next, we prove some dissipative estimates for the weak solutions to
\eqref{1b}--\eqref{3b}.
\begin{lemma}
\label{abs} Suppose $n=2$. For $\bov_0\in \dot H$, $\varphi_0\in
\dot H^2$, any weak solution of \eqref{1b}--\eqref{3b} verifies
\begin{equation}
\label{dissipation} \|
(\bov(t),\vf(t))\|_{\Phi}^2\,\le\,\mathcal{D}(\|(\bov_0,\vf_0)\|_{\Phi})e^{-\alpha
t} + C,
\end{equation}
where the positive constants $C$ and $\alpha$ are independent on the
solution and depend only on the coefficients of the system.
\end{lemma}

\begin{proof}
Multiplying \eqref{3b} with $\varphi$ and integrating over $\T^2$,
we get
 \be
 \frac12\frac{d}{dt} \|\varphi\|^2+K\lambda \|\Delta \varphi\|^2+\frac{\lambda}{\epsilon^2}\int_{\T^2}|\nabla \varphi|^4 dx
 =-\int_{\T^2} (\bov\cdot \nabla) \varphi \varphi dx+
 \frac{\lambda}{\epsilon^2}\int_{\T^2}|\nabla \varphi|^2 dx.\label{ab1}
 \ee
 The righthand side of \eqref{ab1} can be estimated as follows
 \bea
 \frac{\lambda}{\epsilon^2}\int_{\T^2}|\nabla \varphi|^2 dx
 &\leq&
 \frac{\lambda}{4\epsilon^2}\int_{\T^2}|\nabla \varphi|^4 dx+
 \frac{\lambda}{\epsilon^2}|Q|,\non
 \\
 -\int_{\T^2} (\bov\cdot \nabla) \varphi \varphi dx\leq \|\bov\|\|\nabla
 \varphi\|_{L^4}\|\varphi\|_{L^4}
 &\leq& \frac{1}{2\delta_1} \|\bov\|^2+ \frac{\delta_1}{2}\|\nabla
 \varphi\|_{L^4}^2\|\varphi\|_{L^4}^2,\non
 \eea
 $\delta_1>0$ is a small constant to be determined later. For
 $\bov\in\dot
 V$, we infer from the Poincar\'e inequality that
 $ \|\bov\|\leq C_P\|\nabla \bov\|$, where the constant $C_P>0$ depends only on $\T^2$. For $\varphi\in  \dot H^2$, we infer
 from the
 Sobolev embedding theorem, Poincar\'e
 inequality and H\"older inequality that
 \be |\T^2|^{-\frac14}\|\varphi \|\leq \|\varphi\|_{L^4}\leq C_1\|\nabla \varphi\|\leq
 C_1|\T^2|^\frac14\|\nabla \varphi\|_{\mathbf{L}^4},\non
 \ee
 where $C_1$ is constant depending only on $\T^2$.
 As a result,
 \bea
 && \frac{\delta_1}{2}\|\nabla
 \varphi\|_{\mathbf{L}^4}^2\|\varphi\|_{L^4}^2\leq
 \frac{\delta_1}{2}C_1^2|\T^2|^\frac12\int_{\T^2}|\nabla \varphi|^4dx,\non\\
 && \|\varphi\| ^2\leq C_1^2|\T^2|\|\nabla \varphi\|^2_{\mathbf{L}^4}\leq
 \frac{\lambda}{4\epsilon^2}\int_{\T^2}|\nabla
 \varphi|^4dx+\frac{\epsilon^2C_1^4|\T^2|^2}{\lambda}.\non
 \eea
 Hence, we deduce that
 \bea
 && \frac12\frac{d}{dt} \|\varphi\|^2+K\lambda \|\Delta \varphi\|^2
 +\left(\frac{\lambda}{2\epsilon^2}-\frac{\delta_1}{2}C_1^2|\T^2|^\frac12\right)\int_{\T^2}|\nabla
 \varphi|^4
 dx+\|\varphi\|^2\non\\
 &\leq& \frac{C_P}{2\delta_1} \|\nabla \bov\|^2+ \frac{\lambda}{\epsilon^2}|\T^2|+\frac{\epsilon^2C_1^4|\T^2|^2}{\lambda}.
 \label{ab1a}
 \eea
 Multiplying \eqref{ab1a} by $\delta_2>0$ and adding it to the basic
 energy law \eqref{energy}, we obtain
 \bea
  && \frac{d}{dt}\left[\frac{1}{2}\|\bov\|^2+\frac{K}{2}\|\Delta
\varphi\|^2+\int_{\T^2}F(\d)\,dx+\frac{\delta_2}{2}\|\varphi\|^2\right]+\int_{\T^2}\left[\mu_1(D_{kp}d_{k}d_{p})^2+
2\mu_5|D\d|^2 \right]dx\non\\
 \non\\
&& +\lambda \left\|-K\Delta^2
 \varphi+\nabla\cdot f(\nabla \vp)\right\|^2+\left(\frac{\mu_4}{2}-\frac{\delta_2C_P}{2\delta_1}\right)\|\nabla \bov\|^2+\delta_2
 K\lambda \|\Delta
 \varphi\|^2\non\\
 &&
 +\delta_2\left(\frac{\lambda}{2\epsilon^2}-\frac{\delta_1}{2}C_1^2|\T^2|^\frac12\right)\int_{\T^2}|\nabla
 \varphi|^4
 dx+\delta_2\|\varphi\|^2\non\\
 &\leq& \delta_2\left(\frac{\lambda}{\epsilon^2}|\T^2|+\frac{\epsilon^2C_1^4|\T^2|^2}{\lambda}\right).\label{ab2}
 \eea
 Take $\delta_1,\delta_2$ that satisfying
 \be
 \delta_1=\frac{\lambda}{2\epsilon^2C_1^2|\T^2|^\frac12},\quad
 \delta_2=\frac{\mu_4\delta_1}{2C_P}=\frac{\lambda\mu_4}{4\epsilon^2C_PC_1^2|\T^2|^\frac12}.\non
 \ee
 We deduce from \eqref{ab2} that
 \begin{eqnarray}
  && \frac{d}{dt}\left[\frac{1}{2}\|\bov\|^2+\frac{K}{2}\|\Delta
\varphi\|^2+\int_{\T^2}F(\d)\,dx+\frac{\delta_2}{2}\|\varphi\|^2\right]
+\frac{\mu_4}{4}\|\nabla \bov\|^2\non\\
 && \quad +\delta_2
 K\lambda \|\Delta
 \varphi\|^2+\frac{\delta_2\lambda}{4\epsilon^2}\int_{\T^2}|\nabla
 \varphi|^4 dx \leq \delta_2\left(\frac{\lambda}{\epsilon^2}|\T^2|+\frac{\epsilon^2C_1^4|\T^2|^2}{\lambda}\right).\label{ab2a}
 \end{eqnarray}
 Define $ \Psi(t):=\mathcal{E}(t)+\frac{\delta_2}{2}\|\varphi\|^2$. It is easy to see that
 \be
 C_3\left(\|\bov\|^2+\|\Delta \varphi\|^2+ \|\varphi\|^2+ \int_{\T^2}|\nabla \varphi|^4dx
 +1\right)\geq \Psi(t)\geq C_2(\|\bov\|^2+\|\varphi\|_{H^2}^2),\non
 \ee
 where $C_2$, $C_3$ are positive constants depending on $\T^2, K,
 \epsilon,\lambda,\mu_4$ but not on the solution.
 Thus, we can conclude that there exist two positive constants $C_4, C_5$ depending only on $\T, K,
 \epsilon, \lambda,\mu_4$ such
 that
 \bes \frac{d}{dt}\Psi(t)+ C_4\Psi(t)\leq C_5.\non
 \ees
 As a result,
 \bes
 \|\bov(t)\|^2+\|\varphi(t)\|_{H^2}^2\leq \frac{1}{C_2}\Psi(t)\leq
 \frac{1}{C_2}e^{-C_4 t}\Psi(0)+\frac{C_5}{C_2C_4}, \quad \forall\ t\geq 0.\non
 \ees
 The proof is complete.
\end{proof}

\subsection{Higher-order estimates}
Next, we show that the weak solutions turn out to be regular for
strictly positive times. This, will imply the compactness of the
solution operator $S(t)$.
 The following lemma plays an important role in the subsequent proof.
 It is worthwhile noting that, since the coupling in equation \eqref{3b} is weak, this result is valid both
 for $n=2,3$.

 \begin{lemma}\label{A2d}
  Suppose $n=2,3$. We have
 \be \frac{d}{dt} \|\nabla \Delta \vp\|^2 +\lambda K \left\|\nabla \Delta^2\varphi \right\|^2 \leq
 C(\|\nabla \bov\|^2+\|\nabla \Delta \vp\|^2)\|\nabla \Delta \vp\|^2+C\|\nabla \bov\|^2+C,\label{A2da}
 \ee
 where $C$ is a positive constant depending on $\|\bov_0\|$, $\|\vp_0\|_{H^2}$ and coefficients of the system.
 \end{lemma}
 \begin{proof}
 We just work in the $3D$ case and it is easy to verify that the same result holds in
 $2D$.
 Multiplying \eqref{3b} by $\Delta^3 \vp$, integrating over $\T^3$, due to the periodic boundary condition, we have
 \bea && \frac{1}{2}\frac{d}{dt}\|\nabla \Delta \varphi\|^2+\lambda K \left\|\nabla \Delta^2\varphi \right\|^2\non\\
 &=& -\int_{\T^3} \nabla \Delta^2 \varphi\cdot \nabla(\bov\cdot \nabla
 \varphi) dx-\lambda  \int_{\T^3} \nabla \Delta^2\varphi \cdot \nabla \left[\nabla\cdot
 f(\d)\right] dx. \label{dndvp}
 \eea
 By the uniform estimates \eqref{dissipation}, the Agmon inequality and Gagliardo--Nirenberg
 inequality in $3D$, we get
 \bea
 \|\nabla \varphi\|_{L^\infty}&\leq& C\|\vp\|_{H^3}^\frac12\|\vp\|_{H^2}^\frac12\leq C(\|\nabla \Delta
 \varphi\|^\frac12+1),\non\\
  \|\nabla\nabla \varphi\|_{L^3}&\leq& C(\|\nabla
 \Delta \varphi\|^\frac12\|\Delta \vp\|^\frac12+\|\Delta \vp\|)\leq C(\|\nabla \Delta
 \varphi\|^\frac12+1).\non
 \eea
 Now we estimate the right-hand side of \eqref{dndvp} term by term.
 \bea
 \left|\int_{\T^2} \nabla \Delta^2 \varphi \cdot \nabla(\bov\cdot \nabla
 \varphi) dx\right|
 &\leq& C\|\nabla \Delta^2 \varphi\|(\|\nabla \bov\|\|\nabla
 \varphi\|_{\mathbf{L}^\infty}+ \|\bov\|_{\mathbf{L}^6}\|\nabla\nabla \varphi\|_{\mathbf{L}^3})\non\\
 &\leq& C\|\nabla \Delta^2 \varphi\|\|\nabla \bov\|(\|\nabla \Delta
 \varphi\|^\frac12+1)\non\\
 &\leq& \varepsilon\|\nabla \Delta^2 \varphi\|^2+\|\nabla
 \bov\|^2\|\nabla
 \Delta \varphi\|^2+C\|\nabla \bov\|^2.
 \non
 \eea
 \bea
 &&\lambda  \left|\int_{\T^2} \nabla \Delta^2\varphi \cdot \nabla \left[\nabla\cdot
 f(\d)\right] dx\right|=-\frac{\lambda }{\epsilon^2}\int_{\T^2} \nabla \Delta^2\varphi
 \cdot \nabla \left[(3|\nabla
 \varphi|^2-1)\Delta\varphi\right]dx\non\\
 &\leq& C\|\nabla \Delta^2 \varphi\|\|\nabla \Delta \varphi\|+ C\|\nabla \Delta^2
 \varphi\|\|\nabla \varphi\|_{\mathbf{L}^\infty}^2\|\nabla \Delta \varphi\|\non\\
 && +C\|\nabla \Delta^2
 \varphi\|\|\nabla \varphi\|_{\mathbf{L}^\infty}\|\nabla\nabla \varphi\|_{\mathbf{L}^3}\|\Delta
 \varphi\|_{L^6}\non\\
 &\leq& \varepsilon \|\nabla \Delta^2
 \varphi\|^2+ C\|\nabla \Delta \varphi\|^2+  C\|\nabla \Delta \varphi\|^2\|\nabla
 \varphi\|_{\mathbf{L}^\infty}^4+ C\|\nabla\nabla \varphi\|^2_{\mathbf{L}^3}\|\Delta
 \varphi\|^2_{L^6}\|\nabla
 \varphi\|_{\mathbf{L}^\infty}^2\non\\
 &\leq&\varepsilon \|\nabla \Delta^2
 \varphi\|^2+  C\|\nabla \Delta \varphi\|^4+C.
 \non
 \eea
  Taking $\varepsilon=\frac{\lambda K}{4}$, we infer from the above estiamtes that \eqref{A2da} holds.
   The proof is complete.
 \end{proof}
Denote
 \be
 \mathcal{Q}= -K\Delta^2
 \varphi+\nabla\cdot f(\d).\non
 \ee
 By the definition of $\mathcal{Q}$ and the Sobolev embedding theorem, we can easily derive the the following
 estimates.
 \begin{lemma}\label{eqes}
 Suppose $n=2,3$. We have
 $\|\nabla \Delta \vp\|\leq
 C\|\mathcal{Q}\|^\frac12+C$, $ \|\Delta^2 \varphi\|\leq \frac{2}{K}\|\mathcal{Q}\|+C$,
 $\|\nabla \Delta^2 \varphi\|\leq \frac{2}{K}\|\nabla
 \mathcal{Q}\|+C$, where $C$ is a constant depending on $\|\varphi\|_{H^2},K,\epsilon$. Moreover,
 $\| \Delta^3 \varphi\|\leq \frac{2}{K}\|\Delta
 \mathcal{Q}\|+C$, where $C$ is a constant depending on $\|\varphi\|_{H^3},K,\epsilon$.
 \end{lemma}
 Next, we prove the following higher-order estimate for $\vp$:
\begin{lemma} \label{comp}
Suppose $n=2,3$.  For any $\bov_0\in  \dot H$, $\varphi_0\in  \dot
H^2$, the weak solution to \eqref{1b}--\eqref{3b} satisfies
 \be
 \label{pos-t-regu}
 \|\vp(t)\|_{H^3}\leq \frac{1+t}{t}\mathcal{D}(\|(\bov_0,\vf_0) \|_\Phi), \quad \forall t\,>\,0.
 \ee
  Moreover, if we assume in addition that $\vf_0\in H^3$, $\|\vp(t)\|_{H^3}$ can be bounded by a constant depending on
  $\|\bov_0\|$ and $\|\vp_0\|_{H^3}$ uniformly in time.
 \end{lemma}
  \begin{proof}
 We infer from Lemma \ref{low} and Lemma \ref{eqes} that for any
 $r>0$ and $t\geq 0$,
 \be
  \sup_{t\geq 0} \int_t^{t+r} \|\nabla \Delta
 \vp(\tau)\|^2d\tau \leq \sup_{t\geq 0} C\int_t^{t+r}\|\mathcal{Q}(\tau)\|
 d\tau+Cr\leq C \int_0^\infty\|\mathcal{Q}(\tau)\|^2+Cr\leq C(1+r),\label{inter1}
 \ee
 \be
 \sup_{t\geq 0} \int_t^{t+r} \|\nabla \bov(\tau)\|^2d\tau\leq \int_0^{+\infty} \|\nabla \bov(\tau)\|^2d\tau\leq C.\label{inter2}
 \ee
 It follows from \eqref{A2da} and the uniform Gronwall lemma \cite[Lemma III.1.1]{TE} that
 \be
 \|\nabla \Delta \vp(t+r)\|^2\leq C\left(1+\frac{1}{r}\right), \quad \forall t\geq 0,\label{uge}
 \ee
 which together with \eqref{dissipation} yields \eqref{pos-t-regu}.

If we assume that $\varphi_0\in H^3$, then by \eqref{A2da},
\eqref{inter1}, \eqref{inter2} and the
 standard Gronwall inequality, we have
  \bea
   && \|\nabla \Delta \vp(t)\|^2\non\\
    &\leq& \|\nabla \Delta \vp_0\|^2{\rm exp}\left(C\int_0^t (\|\nabla \bov(\tau)\|^2
   +\|\nabla \Delta \vp(\tau)\|^2+1)d\tau\right)\non\\
   && +C\int_0^t (\|\nabla \bov(s)\|^2+1)\ \! {\rm exp}\left(-C\int_0^s (\|\nabla \bov(\tau)\|^2
   +\|\nabla \Delta \vp(\tau)\|^2+1)d\tau\right)ds\non\\
   &\leq& \|\nabla \Delta \vp_0\|^2{\rm exp}\left(C\int_0^1 (\|\nabla \bov(\tau)\|^2
   +\|\nabla \Delta \vp(\tau)\|^2+1)d\tau\right)+C\int_0^1 (\|\nabla \bov(s)\|^2+1)ds\non\\
   &\leq& C,\quad \forall \ t\in [0,1],\non
   \eea
 where $C$ is a constant depending on $\|\bov_0\|$, $\|\vp_0\|_{H^3}$. Taking $r=1$
 in \eqref{uge} and using \eqref{dissipation}, we obtain the uniform
 estimate on $\|\vp(t)\|_{H^3}$ for all $t\geq 0$. The proof is complete.
  \end{proof}
  By the Sobolov embedding theorem, we easily deduce the follow result
 \bc
 \label{ninf}
 Suppose  $n=2,3$. For any $\bov_0\in  \dot H$, $\varphi_0\in  \dot H^2$,  we have
 \be
 \label{vpin}
 \|\nabla \vp(t)\|_{\mathbf{L}^\infty}\leq \frac{1+t}{t}\mathcal{D}(\|(\bov_0,\vf_0) \|_\Phi), \quad \forall t\,>\,0.
 \ee
  Moreover, if we assume in addition that $\vf_0\in H^3$, $\|\nabla \vp(t)\|_{\mathbf{L}^\infty}$
  can be bounded by a constant depending on  $\|\bov_0\|$ and $\|\vp_0\|_{H^3}$ uniformly in time.
 \ec
 Using Corollary \ref{ninf}, we are able to derive the higher-order energy inequality in $2D$.
 \begin{lemma}\label{A22d}
 Suppose $n=2$. Let $$\mathbf{A}(t)=\|\nabla
 \bov(t)\|^2+\alpha\left\|\mathcal{Q}(t)\right\|^2,$$
  where $\alpha>0$ is a small constant to be chosen later (cf. \eqref{alpha} below). We have
 \be \frac{d}{dt}\mathbf{A}(t)+\frac{\mu_4}{4}\|\Delta \bov\|^2
 +\frac{\alpha\lambda K}{2}\|\Delta \mathcal{Q}\|^2\leq C(\mathbf{A}^2(t)+\mathbf{A}(t)), \quad \forall\ t\geq t_1>0,
 \label{A22da}
 \ee
 where $t_1>0$ is arbitrary and $C$ is a constant depending on $\|\bov_0\|$,
 $\|\varphi_0\|_{H^2}$ and $t_1$. Moreover,
 if we assume that $\varphi_0\in
 H^3$, \eqref{A22da} holds for $t\geq 0$ with $C$ being dependent of $\|\bov_0\|$,
 $\|\varphi_0\|_{H^3}$.
 \end{lemma}
 \begin{proof}
 Recall the computation in \cite[pp. 1475]{CG10} that
 $ \nabla \cdot \tilde{\sigma}^e = -(\nabla \cdot f(\d)) \d -\nabla F(\d) +K\Delta ^2\vp \nabla \vp
 -K\nabla \left(\frac{|\nabla \vp|^2}{2}\right) $.
 We note that \eqref{1b} can be written in the following form
 \be
 \bov_t+ \bov\cdot\nabla \bov-\frac{\mu_4}{2} \Delta \bov+\nabla P
=\nabla \cdot \tilde{\sigma}^d+  (K\Delta ^2\vp-\nabla \cdot
f(\d))\d,\label{1b1}
  \ee
where $P=p+\nabla \left(\frac{K|\nabla \vp|^2}{2}+F(\d)\right)$.
Using \eqref{1b1}, we have
 \bea
 && \frac{1}{2}\frac{d}{dt}\|\nabla \bov\|^2= -\int_{\T^2} \bov_t\cdot
 \Delta \bov dx\non\\
 &=& \int_{\T^2} (\bov\cdot\nabla) \bov\cdot \Delta \bov dx - \frac{\mu_4}{2}\|\Delta \bov\|^2
 - \mu_1\int_{\T^2}[\nabla\cdot((\d^\top D(\bov)\d)\d\otimes\d)] \cdot \Delta \bov dx \nonumber\\
&&- \mu_5\int_{\T^2}[\nabla\cdot((D(\bov)\d\otimes\d+\d\otimes
D(\bov)\d))]\cdot \Delta \bov dx\non\\
&& -\int_{\T^2}[(K\Delta ^2\vp-\nabla \cdot f(\d))\d] \cdot \Delta
\bov dx.\label{DV}
 \eea
 Using the periodic boundary condition and integration by parts, the
 right-hand side of \eqref{DV} can be manipulated as follows
 \bea
 && -\mu_1\int_{\T^2}\nabla\cdot[(\d^\top D(\bov)\d)\d\otimes\d] \cdot \Delta \bov
 dx=-\mu_1\int_{\T^2}\nabla_j(d_kD_{kp}d_pd_id_j)\nabla_l\nabla_lv_i
 dx\non\\
 &=& -\mu_1\int_{\T^2}\nabla_l(d_kD_{kp}d_pd_id_j)\nabla_l\nabla_jv_i
 dx=
 -\mu_1\int_{\T^2}\nabla_l(d_kD_{kp}d_pd_id_j)\nabla_lD_{ij}dx\non\\
 &=& -\mu_1\int_{\T^2} (d_kd_p\nabla_l D_{kp})^2
 dx-2\mu_1\int_{\T^2}
D_{kp}\nabla_l d_kd_pd_id_j\nabla _l D_{ij} dx\non\\
 && -2\mu_1\int_{\T^2} D_{kp}d_kd_pd_i\nabla_l d_j\nabla_lD_{ij}
 dx:= -\mu_1\int_{\T^2} (d_id_j\nabla_l D_{ij})^2 dx+ I_1+I_2.\non
 \eea
 \bea
 && - \mu_5\int_{\T^2}\nabla\cdot[(D(\bov)\d\otimes\d+\d\otimes
D(\bov)\d)]\cdot \Delta \bov dx\non\\
&=& - \mu_5\int_{\T^2}\nabla_j(D_{ik}d_kd_j)\nabla_l\nabla_lv_i dx-
\mu_5\int_{\T^2} \nabla_i(d_jD_{ik}d_k )\nabla_l\nabla_lv_j dx\non\\
&=& - \mu_5\int_{\T^2}\nabla_l(D_{ik}d_kd_j)\nabla_l\nabla_jv_i dx-
\mu_5\int_{\T^2} \nabla_l(d_jD_{ik}d_k )\nabla_l\nabla_iv_j dx
\non\\
&=& - 2\mu_5\int_{\T^2}\nabla_l(D_{ik}d_kd_j)\nabla_lD_{ij} dx\non\\
&=& - 2\mu_5\int_{\T^2}(\nabla_lD_{ik}d_k)^2 dx-
2\mu_5\int_{\T^2}\nabla_l d_jd_kD_{ik}\nabla_l D_{ij} dx
-2\mu_5\int_{\T^2}
d_j\nabla_l d_k D_{ik}\nabla_lD_{ij} dx\non\\
 &:=& - 2\mu_5\int_{\T^2}(\nabla_lD_{ik}d_k)^2 dx+I_3+I_4.\non
 \eea
 \be
 -\int_{\T^2}[(K\Delta ^2\vp-\nabla \cdot
f(\d))\d] \cdot \Delta \bov dx:=I_5.\non
 \ee
 Summing up, we have
 \bea
 &&\frac{1}{2}\frac{d}{dt}\|\nabla \bov\|^2+ \frac{\mu_4}{2}\|\Delta \bov\|^2+\mu_1\int_{\T^2} (d_id_j\nabla_l D_{ij})^2 dx
 +2\mu_5\int_{\T^2} (\nabla_l D_{ik}d_k)^2 dx\non\\
&=& \int_{\T^2} (\bov\cdot\nabla)\bov\cdot \Delta \bov
dx+\sum_{k=1}^5 I_k.\label{dnv}
 \eea
 By Lemma
\ref{comp} and Corollary \ref{ninf}, for any $t_1>0$, we have
obtained the uniform estimate:
 \be
 \|\vp(t)\|_{H^3}+ \| \nabla \vp(t)\|_{\mathbf{L}^\infty}\leq M, \quad \forall t\geq
 t_1>0. \label{bdvp3}
 \ee
 We now apply the Gagliardo--Nirenberg inequality, Lemma \ref{eqes} and
\eqref{bdvp3} to estimate the right-hand side of \eqref{dnv}.
 \be
   \int_{\T^2} (\bov\cdot\nabla)\bov\cdot \Delta \bov
dx \leq \|\bov\|_{\mathbf{L}^4}\|\nabla
\bov\|_{\mathbf{L}^4}\|\Delta \bov\|\leq \frac{\mu_4}{8} \|\Delta
\bov\|^2+ C\|\nabla \bov\|^4,\non
 \ee
 Since
 \bea
 && \|\nabla
 \bov\|^2\|\nabla \d\|^2_{\mathbf{L}^\infty}\|\d\|^2_{\mathbf{L}^\infty}\leq C \|\nabla
 \bov\|^2( \|\Delta ^2 \varphi\|+1)\non\\
 &\leq& C \|\nabla
 \bov\|^2( \|\mathcal{Q}\|+1)\leq C\|\mathcal{Q}\|^2+ C \|\nabla
 \bov\|^4+C\|\nabla \bov\|^2,\non
 \eea
 we have
 \bea
 I_1  &\leq& \frac{\mu_1}{4}\int_{\T^2} (d_id_j\nabla_lD_{ij})^2 dx+C \|\nabla
 \bov\|^2\|\nabla \d\|^2_{\mathbf{L}^\infty}\|\d\|^2_{\mathbf{L}^\infty}\non\\
 &\leq& \frac{\mu_1}{4}\int_{\T^2} (d_id_j\nabla_lD_{ij})^2 dx+ C\|\mathcal{Q}\|^2+ C \|\nabla
 \bov\|^4+C\|\nabla \bov\|^2.\non
 \eea
 For $I_2$, after integrating by parts, we have
 \bea
I_2&=&-2\mu_1\int_{\T^2}D_{kp}d_kd_pd_i\nabla_ld_j \nabla_lD_{ij}dx
\non\\
&=&2\mu_1\int_{\T^2}\nabla_lD_{kp}d_kd_pd_i\nabla_ld_j
D_{ij}dx+4\mu_1\int_{\T^2}D_{kp}\nabla_ld_k d_pd_i\nabla_ld_jD_{ij}dx \non\\
&&+2\mu_1\int_{\T^2}D_{kp}d_kd_p\nabla_ld_i\nabla_ld_jD_{ij}
dx+2\mu_1\int_{\T^2}D_{kp}d_kd_pd_i\nabla_l\nabla_l d_jD_{ij}dx
\non\\
&:=& I_{2a}+I_{2b}+I_{2c}+I_{2d},\non
 \eea
 where
 \bea
 I_{2a}&\leq& \frac{\mu_1}{4}\int_{\T^2}(d_id_j\nabla_lD_{ij})^2 dx+C \|\nabla
 \bov\|^2\|\nabla \d\|^2_{\mathbf{L}^\infty}\|\d\|^2_{\mathbf{L}^\infty}\non\\
 &\leq&
 \frac{\mu_1}{4}\int_{\T^2} (d_id_j\nabla_lD_{ij})^2 dx+C\|\mathcal{Q}\|^2+ C \|\nabla
 \bov\|^4+C\|\nabla \bov\|^2,\non\\
 I_{2b}+I_{2c}&\leq &C \|\nabla
 \bov\|^2\|\nabla \d\|^2_{\mathbf{L}^\infty}\|\d\|^2_{\mathbf{L}^\infty}\leq C\|\mathcal{Q}\|^2+ C \|\nabla
 \bov\|^4+C\|\nabla \bov\|^2,\non\\
  I_{2d}&\leq& C\|\d\|_{\mathbf{L}^\infty}^3\|\Delta \d\|\|\nabla \bov\|_{\mathbf{L}^4}^2\leq
   C(\|\Delta^2\vp\|^\frac12+1)\|\Delta \bov\|\|\nabla
  \bov\|\non\\
  &\leq& \frac{\mu_4}{8}\|\Delta
  \bov\|^2+ C\|\mathcal{Q}\|^2+ C \|\nabla
 \bov\|^4 +C\|\nabla \bov\|^2.\non
 \eea
As a consequence,
 \be
 I_2\leq \frac{\mu_1}{4}\int_{\T^2} (d_id_j\nabla_lD_{ij})^2
  dx+\frac{\mu_4}{8}\|\Delta
  \bov\|^2+C\|\mathcal{Q}\|^2+C \|\nabla
 \bov\|^4+C\|\nabla \bov\|^2.\non
 \ee
 Next,
 \begin{eqnarray*}
 I_3+I_4&\leq& \|\Delta \bov\|\|\nabla \bov\|\|\d\|_{\mathbf{L}^\infty}\|\nabla
 \d\|_{\mathbf{L}^\infty}\leq \varepsilon\|\Delta \bov\|^2+C \|\nabla \bov\|^2\|\nabla
 \d\|_{\mathbf{L}^\infty}^2\|\d\|_{\mathbf{L}^\infty}^2\non\\
 &\leq& \frac{\mu_4}{8}\|\Delta \bov\|^2+ C\|\mathcal{Q}\|^2+ C \|\nabla
 \bov\|^4 +C\|\nabla \bov\|^2.
 \end{eqnarray*}
 \be
 I_5 =  \int_{\T^2}\mathcal{Q}\d\cdot \Delta \bov dx\leq \|\Delta \bov\|\|\mathcal{Q}\|\|\nabla
 \varphi\|_{\mathbf{L}^\infty}\non\\
 \leq \frac{\mu_4}{8}\|\Delta \bov\|^2+C\|\mathcal{Q}\|^2.\non
 \ee
 It follows from \eqref{dnv} and the above estimates that
 \bea
 &&\frac{d}{dt}\|\nabla \bov\|^2+ \frac{\mu_4}{2}\|\Delta \bov\|^2+\mu_1\int_{\T^2} (d_id_j\nabla_l D_{ij})^2 dx
 +4\mu_5\int_{\T^2} (\nabla_l D_{ik}d_k)^2 dx\non\\
&\leq & C\|\mathcal{Q}\|^2+ C \|\nabla
 \bov\|^4 +C\|\nabla \bov\|^2, \quad t\geq t_1.\label{dnva}
 \eea
On the other hand, by equation \eqref{3b} and integration by parts,
we have
 \begin{eqnarray}
   \frac{1}{2}\frac{d}{dt}\left\|\mathcal{Q}(t)\right\|^2
   &=&-K\int_{\T^2} \mathcal{Q} \Delta^2\vp_tdx+ \int_{\T^2} \mathcal{Q} (\nabla \cdot f(\d))_tdx \non\\
 &=& -\lambda K\|\Delta \mathcal{Q}\|^2+ K\int_{\T^2} \Delta \mathcal{Q}\cdot \Delta(\bov\cdot \nabla
 \varphi) dx\non\\
 &&+\frac{1}{\epsilon^2}\int_{\T^2} \nabla \mathcal{Q} \cdot [(|\nabla
 \varphi|^2-1)\nabla(\bov\cdot \nabla\varphi)] dx-\frac{\lambda}{\epsilon^2}\int_{\T^2}\nabla \mathcal{Q} \cdot [(|\nabla
 \varphi|^2-1)\nabla \mathcal{Q}] dx\non\\
 && +\frac{2}{\epsilon^2}\int_{\T^2} \nabla
 \mathcal{Q} \cdot [(\nabla \varphi \cdot \nabla(\bov\cdot\nabla
 \varphi))\nabla\varphi] dx-\frac{2\lambda}{\epsilon^2}\int_{\T^2} \nabla
 \mathcal{Q} \cdot [(\nabla \varphi \cdot \nabla
 \mathcal{Q})\nabla\varphi] dx\non\\
 &:=& -\lambda K\|\Delta \mathcal{Q}\|^2+\sum_{k=1}^5 J_k.
 \label{dQt}
 \end{eqnarray}
  The terms $J_1, ..., J_5$ on the right hand side of \eqref{dQt} can be estimated as follows.
 \bea
 J_1&=& K\int_{\T^2} \Delta \mathcal{Q} \Delta \bov\cdot \nabla \varphi dx+ 2K\int_{\T^2} \Delta\mathcal{Q}
 \nabla_k v_i\nabla_k\nabla_i \varphi
 dx +K\int_{\T^2} \Delta \mathcal{Q}\!\  \bov\cdot \nabla \Delta \varphi dx\non\\
 &:=& J_{1a}+J_{1b}+J_{1c},\non
 \eea
 where by the uniform estimate \eqref{bdvp3}, Lemma \ref{eqes} and the Sobolev embedding theorem, we get
 \bea
 J_{1a}&\leq& K\|\nabla \varphi\|_{\mathbf{L}^\infty}\|\Delta \bov\|\|\Delta
 \mathcal{Q}\|\leq \frac{\mu_4}{16\alpha}\|\Delta \bov\|^2+ \frac{4\alpha K^2M^2}{\mu_4}\|\Delta
 \mathcal{Q}\|^2,\non\\
 J_{1b}&\leq& C\|\Delta \mathcal{Q}\|\|\nabla
 \bov\|_{\mathbf{L}^4}\|\varphi\|_{W^{2,4}}\leq C\|\Delta \mathcal{Q}\|\|\nabla
 \bov\|^\frac12\|\Delta \bov\|^\frac12 \non\\
  &\leq& \frac{\lambda K}{8} \|\Delta \mathcal{Q}\|^2+ \frac{\mu_4}{16\alpha}\|\Delta \bov\|^2+C\|\nabla \bov\|^2,\non\\
 J_{1c}&\leq&C\|\Delta \mathcal{Q}\|\|
 \bov\|_{\mathbf{L}^4}\|\varphi\|_{W^{3,4}}\leq C\|\Delta \mathcal{Q}\|\|\nabla \bov\|(\|\nabla \Delta \varphi\|_{\mathbf{L}^4}+1)
 \non\\
 &\leq& \frac{\lambda K}{8}\|\Delta \mathcal{Q}\|^2+C\|\nabla \bov\|^2(\|\mathcal{Q}\|+1)\non\\
 &\leq& \frac{\lambda K}{8}\|\Delta \mathcal{Q}\|^2+ C\|\mathcal{Q}\|^2+ C \|\nabla
 \bov\|^4 +C\|\nabla \bov\|^2.\non
 \eea
 Next,
 \bea
 J_2+J_4&\leq& C\|\nabla \mathcal{Q}\|(\|\nabla
 \varphi\|_{\mathbf{L}^\infty}^2+1)(\|\nabla \bov\|\|\nabla
 \varphi\|_{\mathbf{L}^\infty}+\|\bov\|_{\mathbf{L}^4}\|\varphi\|_{W^{2,4}})\non\\
 &\leq& C(\|\Delta
 \mathcal{Q}\|^\frac12\|\mathcal{Q}\|^\frac12+\|\mathcal{Q}\|)\|\nabla
 \bov\|\non\\
 &\leq& \frac{\lambda K}{8}\|\Delta \mathcal{Q}\|^2+C \|\mathcal{Q}\|^2+C\|\nabla
 \bov\|^2,\non\\
  J_3+J_5&\leq& C\|\nabla \mathcal{Q}\|^2(\|\nabla
 \varphi\|_{\mathbf{L}^\infty}^2+1)\leq \frac{\lambda K}{8}\|\Delta \mathcal{Q}\|^2+C
 \|\mathcal{Q}\|^2.\non
 \eea
 Inserting the above estimates into \eqref{dQt}, we obtain that
\be
   \frac{d}{dt}\left\|\mathcal{Q}(t)\right\|^2+\left(\lambda K-\frac{8\alpha K^2M^2}{\mu_4}\right)\|\Delta
 \mathcal{Q}\|^2-\frac{\mu_4}{4\alpha}\|\Delta \bov\|^2
   \leq C\|\mathcal{Q}\|^2+ C \|\nabla
 \bov\|^4 +C\|\nabla \bov\|^2.\label{dQta}
 \ee
Multiplying \eqref{dQta} by $\alpha$ and adding it to \eqref{dnva},
we have
 \bea
 &&\frac{d}{dt}\mathbf{A}(t)+ \frac{\mu_4}{4}\|\Delta \bov\|^2
  + \alpha\left(\lambda K-\frac{8\alpha K^2M^2}{\mu_4}\right)\|\Delta \mathcal{Q}\|^2\non\\
 &\leq&  C (1+\alpha)(\|\mathcal{Q}\|^2 +\|\nabla \bov\|^4+\|\nabla
 v\|^2), \quad \forall \ t\geq t_1.\label{2Dhigh}
 \eea
Taking
 \be
 \alpha =\frac{\lambda\mu_4}{16KM^2}, \label{alpha}
 \ee
 we conclude from \eqref{2Dhigh} that \eqref{A22da} holds. The lemma is proved.
\end{proof}

\begin{lemma} \label{comph}
Suppose $n=2$. For any $\bov_0\in  \dot H$, $\varphi_0\in H^2$, the
weak solution to \eqref{1b}--\eqref{3b} satisfies
 \begin{equation}
 \label{stima-rego1}
 \|(\bov,\vf)(t)\|_{\Phi_{1}}\,\le\, C(t_2),\quad \forall\  t\geq t_2>0,
\end{equation}
 where $t_2>0$ is arbitrary and $C(t_2)$ is a positive constant depending on
 $\|\bov_0\|$, $\|\vp_0\|_{H^2}$, $t_2$ and coefficients of the system.
 In particular, $\lim_{t_2\to 0^+}C(t_2)=+\infty$.
 If we assume in addition that $\bov_0\in V$ and $\vf_0\in H^4$,
 $\|(\bov,\vf)(t)\|_{\Phi_{1}}$ can be bounded by a constant depending on
  $\|(\bov_0,\vf_0)\|_{\Phi_{1}}$ uniformly in time.
 \end{lemma}
 \begin{proof}
 By Corollary \ref{ninf}, Lemma \ref{low} and the definition of $\alpha$ (cf.  \eqref{alpha}),
 we infer that for arbitrary $t_1>0$,
 \be
 \int_{t_1}^{+\infty}\mathbf{A}(t)dt<+\infty.\label{finiA}
 \ee
 Since \eqref{A22da} holds for $t\geq t_1$, we can apply the uniform Gronwall lemma \cite[Lemma III.1.1]{TE} to
 get the following uniform estimate: for any $r>0$,
 \be
 \mathbf{A}(t+r)\leq C(t_1)\left(1+\frac1r\right), \quad \forall t\geq
 t_1,\non
 \ee
 where $C(t_1)$ is a positive constant depending on $\|\bov_0\|$, $\|\vp_0\|_{H^2}$, $t_1$. Since $t_1$ and $r$
 are arbitrary positive constants, we can prove the uniform estimate \eqref{stima-rego1} for any $t_2>0$.

 If the initial data is more regular, namely, $\bov_0\in V$ and $\vf_0\in H^4$, by Lemma \ref{comp}
 we can easily show that $\|(\bov,\vf)(t)\|_{\Phi_{1}}$ can be uniformly bounded by a constant depending on
  $\|(\bov_0,\vf_0)\|_{\Phi_{1}}$. The proof is complete.
 \end{proof}

 \bc\label{habs}
 Suppose $n=2$. For any $\bov_0\in  \dot H$, $\varphi_0\in H^2$,
 there exists $t^*>0$ depending on $\|\bov_0\|$, $\|\vp_0\|_{H^2}$,
 such that for all $t\geq t^*$, the
 weak solution to \eqref{1b}--\eqref{3b} satisfies
 \begin{equation}
 \label{hab}
 \|(\bov,\vf)(t)\|_{\Phi_{1}}\,\le\, M,\quad \forall\  t\geq t^*,
\end{equation}
where $M$ is independent of $\bov_0,\vp_0$.
 \ec
 \begin{proof}
 It follows from Lemma \ref{abs} that there exists $t_3$ depending
 on $\|\bov_0\|$, $\|\vp_0\|_{H^2}$, such that
 for all $t\geq t_3$, the
 weak solution to \eqref{1b}--\eqref{3b} satisfies
 \begin{equation}
 \label{lab}
 \|(\bov,\vf)(t)\|_{\Phi}\,\le\, M_1,\quad \forall\  t\geq t_3,
\end{equation}
where $M_1$ is independent of $\bov_0,\vp_0$. Now, Lemma \ref{low}
and Lemma \ref{eqes} imply  that for $t\geq t_3$,
 \be
  \sup_{t\geq t_3} \int_t^{t+1} (\|\nabla \bov(\tau)\|^2+\|\nabla \Delta
 \vp(\tau)\|^2)d\tau \leq C \int_{t_3}^\infty (\|\mathcal{Q}(\tau)\|^2+\|\nabla \bov(\tau)\|^2)+C\leq
 C,\non
 \ee
 with $C$ depending on $M_1$. Then \eqref{A2da} and the uniform Gronwall lemma yield that
 $ \|\vp(t+1)\|_{H^3}\leq C$, for $t\geq t_3$. As a consequence,
 $\|\nabla \vp(t)\|_{\mathbf{L}^\infty}\leq M_2$ for all $t\geq
 t_3+1$. For $t\geq t_3+1$, we fix $\alpha$ in \eqref{alpha}
 with $\alpha =\frac{\lambda\mu_4}{16KM_2^2}$. Then \eqref{A22da}
 holds with $C$ only depending on $M_1, M_2$. Applying the uniform
 Gronwall inequality once more, we have $\mathbf{A}(t)\leq M_3$, for
 $t\geq t_3+2$, where $M_3$ depends on $M_1,M_2$. Finally, taking $t^*=t_3+2$,
 we conclude the proof.
 \end{proof}

\noindent \textbf{Proof of Theorem \ref{existence}.} First, thanks
to Lemma \ref{comph}, we see that
  \be
 \mathbf{A}(t)\leq C(t), \quad \forall t\,>\,0,
 \label{AAes}
  \ee
 where $C(t)$ depends on $\|(\bov_0,\vf_0) \|_\Phi$ and $C(t)\nearrow +\infty$ for $t\searrow 0^+$ but
 remains uniformly bounded for $t\nearrow +\infty$.
Integrating \eqref{A22da} over $[t,t+1]$, recalling Lemma \ref{low},
we have
\begin{equation}
\label{stima-rego2} \int_{t}^{t+1}\left(\|\Delta \bov(s)\|^2 +
\|\vf(s)\|^2_{H^6}\right)ds \,\leq \,C(t), \quad \forall \,t\,>\,0.
\end{equation}
As a consequence of \eqref{AAes} and \eqref{stima-rego2}, we
immediately have
$$\int_{t}^{t+1} \left(\|\bov(s)\cdot\nabla \bov(s)\|^2 +\|\bov(s)\cdot \nabla\vf(s)\|^2_{H^2}\right) ds \,\le\,C(t), \quad \forall\, t>0.$$
On the other hand, it is easy to check that \eqref{AAes} and
\eqref{stima-rego2} gives an analogous $L^2(t,
t+1;\mathbf{L}^2(\T^2))$ estimate of $\nabla\cdot
(\tilde{\sigma}^d+\tilde{\sigma}^e)$ for any $t>0$. Hence, by direct
comparison with \eqref{1b}--\eqref{3b}, we can see that
 \begin{equation}
\label{rego-vt} \int_{t}^{t+1}\left(\|\partial_t \bov(s)\|^2 +
\|\partial_t \vf(s)\|^2_{H^2}\right) ds\,\le\, C(t),\quad\forall
\,t\,>\,0.\non
 \end{equation}
The $2D$ smoothing property is thus proved.

  Finally, similar to \cite{LL95}, we know that if the initial data
  are regular, the existence of a weak solution together with high-order estimates
implies a strong solution, and by Theorem \ref{uniqueness},
 the strong solution is actually unique.

\subsection{The global attractor and exponential attractors}

Lemma \ref{abs} and Corollary \ref{habs} entail that there exists a
compact absorbing set in $\Phi$. If we had uniqueness for the weak
solutions, this would be sufficient to prove the existence of the
global attractor by using the classical theory on dynamical systems
(see, e.g., \cite{TE}). We can overcome this difficulty essentially
relying on the regularization of weak solutions to strong solutions
for strictly positive times proved in Theorem \ref{reg}. This
implies that, for strictly positive times, we have enough regularity
to ensure uniqueness by Theorem \ref{uniqueness}. As a consequence,
we have the following weaker form of uniqueness, to which we refer
as \emph{unique continuation}:
 \bp \label{uc} Suppose $n=2$. For any two weak solutions
$(\bov_1,\vf_1)$ and $(\bov_2,\vf_2)$ such that\\
$(\bov_1(T),\vf_1(T))=(\bov_2(T),\vf_2(T))$ at some $T\,>\,0$, then
it holds $(\bov_1,\vf_1)\,\equiv\,(\bov_2,\vf_2)$ for any $t\,\ge\,
T$.
 \ep
  A possible way to construct the global attractor is to
apply the theory of $\ell$-trajectories introduced by  M\'alek and
Ne{\v{c}}as in \cite{MN} and later developed by M\'alek and
Pra\v{z}\'ak in \cite{MP} (For other possible approaches, the reader
is referred to, e.g., Ball \cite{ball97} or to Remark \ref{attr} in
this paper). Besides, we can also use the $\ell$-trajectory method
to study the existence of an exponential attractor.

 For the sake of convenience, we recall some highlight points of the  $\ell$-trajectory method here.
 Roughly speaking,
the $\ell$-trajectory method consists in lifting the dynamics from
the physical phase space to a space of trajectories with an
arbitrary but fixed length $\ell>0$. More precisely, for our current
problem, by $\ell$-trajectory we mean any solution to
\eqref{1b}-\eqref{3b} defined on the time interval $[0,\ell]$. Then,
we  endow the space of $\ell$-trajectories denoted by
$\mathscr{X}_\ell$ with the topology of $L^2(0,\ell;\Phi)$. Note
that weak solutions to \eqref{1b}--\eqref{3b} lie (at least) in
$$C_{w}([0,\ell];\Phi):=\{(\bov,\vp)\in L^\infty(0,\ell;\Phi):\
\langle (\bov,\vp), (\mathbf{u},\psi)\rangle_{\Phi, \Phi'}\in
C([0,\ell]),\ \forall\ (\mathbf{u},\psi)\in \Phi'\},$$ which makes
it reasonable to talk about the point values of trajectories.

The \emph{unique continuation} property implies that from an end
point of an $\ell$-trajectory there starts at most one solution.
Combined with the existence theorem, this implies that if
$(\mathbf{u},\phi)\in \mathscr{X}_\ell$ and $T > \ell$, then there
exists a unique $(\bov, \vp)$ which is a solution to
\eqref{1b}--\eqref{3b} on $[0,T]$ such that $(\mathbf{u},
\phi)=(\bov, \vp)|_{[0,\ell]}$. Then we can define the semigroup
$\mathscr{S}(t)$ on $\mathscr{X}_\ell$:
\begin{equation}
\label{translation}
(\mathscr{S}(t)(\mathbf{u},\phi))(\tau)\,:=\,(\bov(t+\tau),\vf(t +
\tau)),\quad \tau\in[0,\ell].
\end{equation}
 From now on, without loss of generality, we will fix
$\ell=1$. Corollary \ref{habs} implies that there exists $R\,>0$
such that
$$
B_1\,=\,\left\{ (\bov,\vf) \in \Phi_1 \,:\
\|(\bov,\vf)\|_{\Phi_1}\,\le\,R \right\}\subset \Phi_1\subset\subset
\Phi
$$
is a compact, absorbing set for the solution map $S(t)$. Theorem
\ref{uniqueness} entails that the solution operator $S(t)$ confined
on $\mathcal{B}_1$ is indeed a semigroup. Let
\begin{equation}
\label{B1}
\mathcal{B}_1\,:=\,\overline{\bigcup_{t\,\in\,[0,T_0]}S(t)B_1}^{\Phi_1}
\end{equation}
 where $T_0\,>\,0$ is a time such
that $S(t)B_1\,\subset\,B_1$ for all $t\ge T_0$ and the closure is
taken with respect to the weak topology of $\Phi_1$. Then $B_1$ is a
compact, absorbing and positive invariant set for $S(t)$. Define
\begin{equation}\label{Bl}
\mathscr{B}_1^{1}=\{(\mathbf{u}, \phi)\in \mathscr{X}_\ell:
(\mathbf{u}, \phi)(0)\in \mathcal{B}_1\}.
\end{equation}
Note that $\mathscr{B}_1^{1}$ is indeed closed with respect to the
topology of $L^2(0,1;\Phi)$. Using Corollary \ref{habs} and
Proposition \ref{uc}, one can verify that all the assumptions in
\cite[Theorem 2.1]{MP} are satisfied and as a result, the dynamical
system $(\mathscr{S}(t), \mathscr{X}_\ell)$ possesses the global
attractor $\mathbb{A}$. Next, we introduce the following map
evaluation map
 \be
 e:\,L^2(0,1;\Phi)\mapsto \Phi \quad \text{defined by}\
 e((\mathbf{u},\phi))\,=\,(\mathbf{u}(1),\phi(1)). \label{e}
 \ee
 Define $\mathbf{B}=e(\mathscr{B}_1^{1})$. We see that $\mathbf{B}\subset \Phi_1$,
 thus $S(t):\Phi\to \Phi$ is a semigroup on $\mathbf{B}$ and
 $\mathbf{B}$ is positively invariant.
  If we can show that the map $e$ is Lipschitz continuous on
 $\mathscr{B}_1^{1}$
 (which is
indeed true, see \eqref{elip} below), then we can project the global
attractor $\mathbb{A}$ back to the physical space $\Phi$ obtaining
the usual global attractor $\mathcal{A}=e(\mathbb{A})$ for the
dynamic system $(S(t), \mathbf{B})$. Since $\mathbf{B}$ is actually
absorbing, $\mathcal{A}$ is also a global attractor in the phase
space $\Phi$.

\br \label{attr} If one is interested only in the existence of the
global attractor, one can reason as follows, without invoking the
$\ell$-trajectory theory. First of all, combining Theorem
\ref{existence}, \ref{uniqueness} and \ref{reg}, we have that the
restriction of the  solution operator, named $\tilde{S}(t)$, to the
bounded sets of $\Phi_1$ is a semigroup. Moreover, Corollary
\ref{habs} give the dissipativity of $\tilde{S}(t)$ with respect to
the $\Phi_1$ metric. As a consequence, the standard theory of
dynamical systems gives the existence of the global attractor
$\mathcal{A}$ attracting the bounded sets of $\Phi_1$ but with
respect to the $\Phi$-topology. Finally, the smoothing property
implies that $\mathcal{A}$ is indeed the attractor for the weak
solutions, since it attracts also the $\Phi$-bounded sets.
 \er
 Our next step is to
prove the finite dimensionality (in terms of fractal dimension) of
the global attractor $\mathcal{A}$ constructed above and the
existence of an exponential attractor. As anticipated in the
introduction, the finite dimensionality of the global attractor will
be deduced as a consequence of the existence of a finer attracting
set, the exponential attractor. We recall the following (cf.
\cite{EFNT})
\begin{defin}\label{Def2.ea}
A compact subset $\mathcal{M}$ of the phase space $\Phi$ is called
an {\sl exponential attractor}\/ for the semigroup $S(t)$ if the
following conditions are satisfied: {\sl (E1)} The set $\mathcal{M}$
is {\it positively} invariant, i.e.,
 $S(t)\mathcal{M}\subset \mathcal{M}$ for all $t\ge 0$; {\sl (E2)} The fractal dimension (see, e.g., \cite{mane,TE})
 of $\mathcal{M}$ in $\Phi$ is finite; {\sl (E3)} The set $\mathcal{M}$ attracts exponentially
 fast the image of the bounded subsets of the phase space
 $\Phi$. Namely, there exist $C,\beta>0$
 such that
 \begin{equation}\label{expo-attractio}
   \dist_{\Phi}(S(t)B,\mathcal{M})\le C\,e^{-\beta t},\quad \forall \hbox{ bounded set} \ B\subset \Phi,
   \quad\forall t\ge 0.\\[1mm]
 \end{equation}
\end{defin}
Note that, by construction, the exponential attractor, when it
exists,  always contains the global attractor. Thus, property $(E2)$
gives that the global attractor has finite fractal dimension too.
Besides its importance in proving the finite dimensionality of the
global attractor, the existence of an exponential attractor is of
interest in itself. In fact, it may resolve some of the major
drawbacks of the global attractor, namely its arbitrary slow
attraction, which makes the global attractor very sensitive to
perturbation and to numerical approximation, and the difficulty in
estimating its rate of convergence. We refer the readers to the
recent survey \cite{MZsur} for more details and additional
references.

To prove the existence of an exponential attractor $\mathcal{M}$, we
first use the following existence theorem proposed in \cite{EMZ},
which gives an efficient strategy to obtain the existence of an
exponential attractor for the discrete semigroup generated by the
iterations of a proper map $\Bbb S$. Then in a second step, we
construct the desired exponential attractor for the semigroup with
continuous time.
\begin{lemma}
 {\rm (cf. \cite{EMZ})} \label{Lem2.abs}
Let $\mathscr{H}$ and $\mathscr{H}_1$ be two Banach spaces such that
$\mathscr{H}_1$ is compactly embedded into $\mathscr{H}$. Suppose
$\Bbb B_1$ is a bounded closed subset of $\mathscr{H}$. Let us give
a map $ \Bbb S: \Bbb B_1\to\Bbb B_1 $ such that
 \begin{equation}\label{2.sqw}
 \|\Bbb S b_1-\Bbb S b_2\|_{\mathscr{H}_1}\le L\|b_1-b_2\|_{\mathscr{H}},\ \ \forall\ b_1,b_2\in\Bbb B_1,
 \end{equation}
where the constant $L$ is independent of $b_1$ and $b_2$. Then, the
discrete semigroup $\{\Bbb S(n),\, n\in\Bbb N\}$ generated on $\Bbb
B_1$ by the iterations of the map \ $\Bbb S$ possesses an
exponential attractor, i.e., there exists a compact set
 $\mathcal{M}_d\subset\Bbb B_1$ such that
 (E1) $\mathcal{M}_d$ is positively invariant: $\Bbb S\mathcal{M}_d\subset\mathcal{M}_d$;
(E2) The fractal dimension of $\mathcal{M}_{d}$ in $\mathscr{H}$ is
finite: $\dim_f(\mathcal{M}_d,\mathscr{H})\le M<+\infty$;  (E3)
$\mathcal{M}_d$ attracts exponentially the images of $\Bbb B_1$
under the iterations of the map $\Bbb S$: $ \dist_{\mathscr{H}}(\Bbb
S(n)\Bbb B_1,\mathcal{M}_d)\le Ce^{-\kappa n}.$ Moreover, the
positive constants $M$, $C$ and $\kappa$ can be expressed explicitly
in terms of the {\rm squeezing constant} $L$, the size of the set
$\Bbb B_1$ and the entropy of the compact embedding
$\mathscr{H}_1\subset\subset \mathscr{H}$.
\end{lemma}

In order to apply Lemma \ref{Lem2.abs}, one has to properly define
the map $\Bbb{S}$, together with the spaces $\mathscr{H}$,
$\mathscr{H}_1$ and $\mathbb{B}_1$. A typical choice for dissipative
problems like \eqref{1b}--\eqref{3b}, would be (recall \eqref{B1})
$$\Bbb S\,:=\,S(1),\quad \mathscr{H}\,:=\,
\Phi,\quad \mathscr{H}_1\,:=\,\Phi_{1}, \quad \Bbb
B_1\,:=\,\mathcal{B}_1.
$$ Unfortunately, a closer inspection
to system \eqref{1b}--\eqref{3b} reveals that the above choice is
not completely satisfactory in the sense that proving a point-wise
(in time) estimate for the difference of two solutions in the norm
of $\Phi_{1}$ appears to be difficult due to the highly nonlinear
character of the problem. We overcome this difficulty by using the
method of $\ell$-trajectories
 to construct proper spaces $\mathscr{H}$ and
$\mathscr{H}_1$ and then verify the assumptions of Lemma
\ref{Lem2.abs}.
\par
As we did for the construction of the global attractor, we still set
$\ell=1$. Let us define $\Bbb B_1\,:=\,\mathscr{B}_1^{1}$ (recall
\eqref{Bl}) and
 \begin{equation}\label{2.spaces}
 \mathscr{H}:=L^2(0,1; \Phi),\ \ \
 \mathscr{H}_1:=L^2(0,1;\Phi_{1})\cap( W^{1,1}(0,1;V')\,\times\,H^{1}(0,1; L^2)).
 \end{equation} It follows from the Aubin--Lions compactness
lemma that the embedding $\mathscr{H}_1\subset \mathscr{H}$ is
compact. We will apply Lemma \ref{Lem2.abs} to the map $\Bbb
S=\mathscr{S}(1)$ (see \eqref{translation}) acting on the set $\Bbb
B_1$. To this end, we only need to check the smoothing property
\eqref{2.sqw}. All the results in this subsection holds only in the
two dimensional case. Moreover, we do not need any particular
restriction on the values of the structural constants in the
equations. Nevertheless, it would be quite interesting and important
to find an explicit (and possibly sharp) dependence of the fractal
dimension of the attractor with respect to the coefficients in the
equations.

 \begin{lemma}
\label{key-lemma} Suppose  $n\,=\,2$. Let
 $(\bov_1,\vf_1)$ and $(\bov_2,\vf_2)$ be
 two solutions to problem \eqref{1b}--\eqref{3b} with initial conditions in $\mathcal{B}_1$. Denote
 $\bv\,:=\,\bov_1-\bov_2$ and $\bvf\,:=\,\vf_1-\vf_2$. Then, the following estimate holds
 \be
 \label{key-estimate}
 \|\partial_t \bv\|_{L^{1}(1,2;V')}
 + \|\partial_t \bvf\|_{L^{2}(1,2; L^2)} + \|(\bv,\bvf) \|_{L^2(1,2;\Phi_{1})} \,\le\,C\|(\bv,\bvf )\|_{L^2(0,1; \Phi)}.
 \ee
 \end{lemma}
\begin{proof}
We know from  Lemma \ref{comph}  that for $i\,=\,1,2$
 \be
 \|(\bov_i,
\vp_i)(t)\|_{\Phi_1}\leq C, \quad \forall\ t\geq 0.\label{uie}
 \ee Then we test the
equation for $\bv$ with $\bv$ and the equation for $\bvf$ with
$\Delta^2\bvf$, respectively. We obtain
 \bea
 && \frac{1}{2}\frac{d}{dt}(\|\bv \|^2 + \|\Delta \bvf\|^2) +\frac{\mu_4}{2}\|\Delta \bv\|^2+ \lambda K\| \Delta^2 \bvf\|^2\non\\
 &=& -\int_{\T^2}(\bv \cdot \nabla) \bov_2 \cdot \bv\,dx
 - \int_{\T^2} (\tilde{\sigma}^d_1-\tilde{\sigma}^d_2)\cdot \nabla\bv \,dx
  - \int_{\T^2} (\tilde{\sigma}^e_1-\tilde{\sigma}^e_2)\cdot \nabla\bv \,dx\nonumber\\
 &&- \int_{\T^2}(\bv \nabla\vf_1 + \bov_2\nabla\bvf)\Delta^2\bvf \,dx + \lambda \int_{\T^2}\nabla\cdot
 (f(\d_1)-f(\d_2))\Delta^2\bvf\,dx := \sum_{i=1}^5K_i. \label{diff-sol}
 \eea
 Using estimate \eqref{uie}, it is not difficult to see that
 \bea
 K_1&\le&\|\bv\|^2_{\mathbf{L}^4}\| \nabla
\bov_2\|\,\le\,\epsi\|\nabla\bv\|^2 + C(\epsi)\|\bv\|^2,\label{I1}\\
K_4&\leq& \|\bv\|\|
\nabla\vf_1\|_{\mathbf{L}^\infty}\|\Delta^2\bvf\|+
\|\bov_2\|_{\mathbf{L}^4}\|\nabla\bvf\|_{\mathbf{L}^4}\|
\Delta^2\bvf\|\non\\
&\leq& \epsi\| \Delta^2 \bvf\|^2+ C(\epsi)(\|\bv\|^2 + \|
\bvf\|^2_{H^2}),\label{K4}\\
 K_5&\leq& \|\nabla\cdot(f(\d_1)-f(\d_2))\|\| \Delta^2 \bvf\|\non\\
 &\leq& C[(\|\vp_1\|_{L^\infty}^2+1)\|\Delta \bar\vp\|+\|\Delta
\vp_2\|_{L^6}\|\nabla \vp_1+\nabla \vp_2\|_{\mathbf{L}^6}\|\nabla
\bar\vp\|_{\mathbf{L}^6}]\| \Delta^2 \bvf\|\non\\
&\leq&  \epsi\| \Delta^2 \bvf\|^2+ C(\epsi)\|
\bvf\|^2_{H^2}.\label{K5}
 \eea
 To estimate $K_2$, we need to control
$\|\tilde{\sigma}^d_1-\tilde{\sigma}^d_2\|$.
 \bea
\tilde{\sigma}^d_1-\tilde{\sigma}^d_2&=&
\mu_1[(\d_1^{\top}D(\bov_1)\d_1)\d_1\otimes \d_1
-(\d_2^{\top}D(\bov_2)\d_2)\d_2\otimes \d_2]\non\\
&&
+\mu_5(D(\bov_1)\d_1\otimes\d_1-D(\bov_2)\d_2\otimes\d_2)+\mu_5(\d_1\otimes
 D(\bov_1)\d_1- \d_2\otimes
 D(\bov_2)\d_2) \non\\
&:=& J_1+J_2+J_3.\non
 \eea
 We give a detailed $L^2$-estimate only
for the terms $J_1$ in the above decomposition, since for the other
two (lower-order) terms $J_2, J_3$, the argument is essentially the
same and actually simpler. We have
\begin{eqnarray}
 J_1 &:=& (\d_1^{\top}D(\bov_1)\d_1)\d_1\otimes \d_1
-(\d_2^{\top}D(\bov_2)\d_2)\d_2\otimes \d_2
 = \d_1^{\top}D(\bv)\d_1 \d_1\otimes\d_1\nonumber\\
  &&+  (\d_1^{\top}-\d_2^{\top})D(\bov_2)\d_1\d_1\otimes\d_1
  + \d_2^{\top}D(\bov_2)\Big(\d_1\d_1\otimes\d_1 - \d_2\d_2\otimes\d_2\Big)\nonumber\\
 & = & J_{1a} + J_{1b} + J_{1c}.\label{J1}
 \end{eqnarray}
The term $J_{1a}$ multiplied with $\nabla\bv$ produces a nonnegative
(hence negligible in the estimate) term since
\begin{eqnarray*}
J_{1a}:\nabla\bv &= & J_{1a}:D(\bv)\,=\,(\d_1^{\top}D(\bv)\d_1 \d_1\otimes\d_1):D(\bv)\nonumber\\
&=& \d_1^{\top}(D(\bv))\d_1 ( \d_1\otimes\d_1):D(\bv))\,=\, \vert
\d_1^{\top}D(\bv)\d_1 \vert^2\,\ge\,0.
\end{eqnarray*}
Then using Sobolev embedding theorem and \eqref{uie}, we have
 \bea
 \| J_{1b}\|&\le& \|\d_1^{\top}-\d_2^{\top}\|_{\mathbf{L}^\infty}\|D(\bov_2)\|
 \|\d_1(\d_1\otimes\d_1)\|_{\mathbf{L}^\infty}
 \le C\|\bvf\|_{H^3}\leq C\|\bvf\|_{H^4}^\frac12\|\bvf\|_{H^2}^\frac12,
 \label{J12}\\
 \|J_{1c}\|&\leq
&\|\d_2\|_{\mathbf{L}^\infty}\|D(\bov_2)\|(\|\bar
\d\d_1\otimes\d_1\|_{\mathbf{L}^\infty}+\|\d_2\bar
\d\otimes\d_1\|_{\mathbf{L}^\infty}
+ \|\d_1\d_2\otimes\bar\d\|_{\mathbf{L}^\infty}) \nonumber\\
&\le & C\|\bvf\|_{H^3}\leq
C\|\bvf\|_{H^4}^\frac12\|\bvf\|_{H^2}^\frac12. \label{J121}
 \eea
As a consequence,
 \be
 K_2\leq \|\nabla
 \bar\bov\|\|\tilde{\sigma}^d_1-\tilde{\sigma}^d_2\|\leq \varepsilon\|\nabla
 \bar\bov\|^2+ \varepsilon \|\bvf\|^2_{H^4}+C\|\bvf\|_{H^2}^2.\non
 \ee
Concerning $K_3$, we have
 \bea
 \tilde{\sigma}^e_1-\tilde{\sigma}^e_2 &=&K[\nabla(\nabla \cdot\d_1)\otimes\d_1-\nabla(\nabla \cdot\d_2)\otimes\d_2]-K[(\nabla
\cdot\d_1)\nabla \d_1-(\nabla \cdot\d_2)\d_2]\non\\
&& -[f(\d_1)\otimes\d_1-f(\d_2)\otimes\d_2]:=J_4+J_5+J_6.\non
 \eea
 By the Sobolev embedding and \eqref{uie}, we obtain
 \bea
J_4  & = & \|\nabla(\nabla\cdot\d_1)\otimes \d_1 -
\nabla(\nabla\cdot\d_2)\otimes \d_2 \|\le
 \|\bar\d\|_{\mathbf{H}^2}\| \d_1 \|_{\mathbf{L}^\infty} + \|\d_2\|_{\mathbf{H}^2}\|\bar\d \|_{\mathbf{L}^\infty}\nonumber\\
&\le & C\| \bvf\|_{H^3}\leq
C\|\bar\vp\|_{H^4}^\frac12\|\bar\vp\|_{H^2}^\frac12,\non \\
 J_5  & = & \|(\nabla\cdot\d_1)\otimes \nabla \d_1 - (\nabla\cdot\d_2)\otimes \nabla\d_2 \|\le
 \|\bar\d\|_{\mathbf{H}^1}\| \nabla\d_1 \|_{\mathbf{L}^\infty} + \|\d_2\|_{\mathbf{H}^1}\|\nabla \bar\d \|_{\mathbf{L}^\infty}\nonumber\\
&\le & C\| \bvf\|_{H^2}+C\| \bvf\|_{H^3}\leq
C\|\bar\vp\|_{H^4}^\frac12\|\bar\vp\|_{H^2}^\frac12,\non\\
J_6&\leq& C\|\bar\vp\|_{H^2},\non
 \eea
which imply that
 \be K_3\leq \|\nabla
 \bar\bov\|\|\bar\vp\|_{H^4}^\frac12\|\bar\vp\|_{H^2}^\frac12\leq \varepsilon\|\nabla
 \bar\bov\|^2+
 \varepsilon\|\bar\vp\|_{H^4}^2+C\|\bar\vp\|_{H^2}^2.\non
 \ee
Now we test the equation for $\bvf$ by $\bvf+\bar\vp_t$. Similar
computations give
\begin{equation}
\label{diff-sol2} \frac{d}{dt}(\| \bvf\|^2 +\| \Delta\bvf\|^2)+ \|
\bvf_t\|^2+\| \bvf\|^2_{ H^2}\, \le\, C(\|\bv\|^2 + \|\bvf
\|^2_{H^2}).
\end{equation}
Summing \eqref{diff-sol} with \eqref{diff-sol2}, choosing  $\epsi$
sufficiently small, we obtain
 \be
 \frac{d}{dt}(\|\bv \|^2 +  2\| \Delta \bvf\|^2+ \| \bvf\|^2) + \|
\nabla \bv \|^2 + \|\partial_t\bvf \|^2 + \| \bvf\|^2_{H^4}  \le
C(\|\bv \|^2 + \| \bvf\|^2_{H^2}).\label{diff-sol4}
 \ee
  By the Gronwall Lemma, for any $0\,\le\,y-t\,\le\,2$,
 \be
 \| \bv(y) \|^2 + \| \bvf(y)\|^2_{H^2} \le
Ce^{2C}\left( \|\bv (t)\|^2 + \| \bvf(t)\|^2_{ H^2} \right).
\label{diff-sol5}
 \ee Taking $t\,\in\,[0,1]$ and integrating
\eqref{diff-sol4} over $[t,2]$, we infer from \eqref{diff-sol5} that
 \bea
&& \|(\bv,\bvf)(2)\|^{2}_{\Phi} + \int_{t}^{2}(\|(\bv,\bvf)(r)\|^{2}_{\Phi_{1}}+\|\partial_t \bvf(r)\|^2) dr \non\\
&\le & C\left(\|(\bv,\bvf)(t)\|^{2}_{\Phi} + \int_{t}^2
\|(\bv,\bvf)(r)\|^{2}_{\Phi}dr\right)\le C\|
(\bv,\bvf)(t)\|^2_{\Phi}.  \label{diff-sol6}
 \eea
 Integrating \eqref{diff-sol6} with respect to $t$
over $[0,1]$, we finally obtain
 \bea
\|(\bv,\bvf)\|^{2}_{L^2(1,2;\Phi_1)} + \|\partial_t\bvf
\|^2_{L^2(1,2; L^2)} &\le & \int_0^1\int_t^2 \|(\bv,\bvf)(r)\|^{2}_{\Phi_{1}}+\|\partial_t \bvf(r)\|^2) drdt  \non\\
&\leq&
 C\|(\bv,\bvf) \|^2_{L^2(0,1;\Phi)}. \label{diff-sol7}
 \eea
It remains to estimate $\|\partial_t\bv\|_{L^{1}(1,2;V')}$. We use a
duality argument. First, we recall that, for $u\in L^1(1,2;V')$,
$$ \| u\|_{L^1(1,2;V')} = \sup_{\varphi}\Big\vert\int_{1}^{2}\;\langle u,\varphi\rangle dr\Big\vert,$$
where the $\sup$ is taken over the function  $\phi \in L^\infty(1,2;
V)$ such that $\|\phi\|_{L^\infty(1,2; V)} = 1$ and the duality
pairing is  between $V'$ and $V$. Consequently, thanks to
\eqref{uie} and \eqref{diff-sol7}, there holds
\begin{eqnarray*}
\int_{1}^2\|\partial_t\bv(r) \|_{V'}dr
&\le&\int_{1}^2\|\nabla\bv(r)\|dr + \int_{1}^2(\|\bv (r)\|\|\nabla \bov_1(r)\| + \|\nabla\bv(r)\|\|\bov_2(r)\|)dr\nonumber\\
&& + \int_{1}^2(\|\tilde{\sigma}^d_1(r)-\tilde{\sigma}^d_2(r)\|+\|\tilde{\sigma}^e_1(r)-\tilde{\sigma}^e_2(r)\|)dr\non\\
&\le& C\left(\int_{0}^1\|(\bv,\bvf)(r)\|^2_{\Phi}dr\right)^\frac12.
\end{eqnarray*}
The proof is complete.
\end{proof}
Thanks to Lemma \ref{key-lemma}, we have verified that the map $\Bbb
S:=\mathscr{S}(1)$ satisfies all of the assumptions of the abstract
result Lemma \ref{Lem2.abs}. Therefore, the discrete semigroup
 $\{\mathscr{ S}(n),\ n\in\Bbb N\}$ possesses
an exponential attractor $\Bbb{M}_d$ in the trajectory space $\Bbb
B_1$ endowed with the topology of $\mathscr{H}=L^2(0,1;\Phi)$.

Multiplying \eqref{diff-sol4} with $t$ and using the Gronwall lemma,
we easily obtain that
\begin{equation}
\| (\bv,\bvf)(1)\|^2_{\Phi} \,\le\,C\int_{0}^1\|(\bv,\bvf)(r) \|^2
_{\Phi}dr, \label{elip}
\end{equation}
which means that the map $e$ (cf. \eqref{e} for the definition) is
Lipschitz continuous on $\Bbb B_1$.
 This yields
that projecting $\mathbb{M}_d$ back to the phase space $\Phi$ via
 \begin{equation}
 \label{2.prexp}
 \mathcal{M}_d:=e(\mathbb{M}_d)=\mathbb{M}_d\big|_{t=1},
 \end{equation}
the resulting $\mathcal{M}_d$ is indeed the exponential attractor
for the discrete semigroup $\{S(n),\,n\in\Bbb N\}$ acting on
$\mathbf{B}=e(\Bbb B_1)$ (endowed with the topology of $\Phi$).

We note that for all the trajectories $(\bov,\vf)$ starting from
$\mathcal{B}_1$, there holds
 $$\int_{0}^{1}\|\partial_t \bov(s)\|^2 + \|\partial_t \vf(s)\|^2_{H^2} ds\,\le\,C.$$
This, together with \eqref{diff-sol5} imply that the map
$(t,(\bov_0,\vf_0))\,\mapsto\,S(t)((\bov_0,\vf_0))$ is Lipschitz
continuous on $[0,1]\,\times\,\mathcal{B}_1$ with respect to the
$\mathbb{R}\,\times\Phi$ metric. Thus, the desired exponential
attractor $\mathcal{M}$ with continuous time (and on the whole phase
space since $\mathcal{B}_1$ is absorbing) is given by the standard
expression (see \cite{EMZ1} and \cite{MP} for further information)
$$ \mathcal{M}\,:=\,\bigcup_{t\in[0,1]}S(t)\mathcal{M}_d. $$
The proof of Theorem \ref{2datt} is complete.

\section{Convergence to Equilibrium in $2D$}
\label{single-traj-anal}
 \setcounter{equation}{0}
  Theorem \ref{energy1} indicates that the total energy $\mathcal{E}(t)$ (cf.
  \eqref{E}) is decreasing with respect to time, consequently, it serves as a global Lyapunov functional for
system \eqref{1b}--\eqref{3b}. The $\omega$-limit set of
$(\bov_0,\vp_0) \in \dot H\times \dot H^2$ is defined as follows:
 \bea
 \omega(\bov_0,\vp_0)& =  & \{(\bov_\infty, \vp_\infty)\in  \dot H\times \dot H^2:\ \exists
 \  \{t_i\}_{i=1}^\infty\nearrow +\infty,\ \text{such that} \non\\
 &&\ (\bov(t_i),\vp(t_i))\to (\bov_\infty, \vp_\infty)\  \text{in}\ \dot H\times H^2\}.
 \eea
It follows from Lemma \ref{comph} and the well-known result on
dynamical systems (cf. \cite[Lemma I.1.1]{TE}) that
 \bl
 $\omega(\bov_0, \vp_0)$ is a non-empty bounded connected
subset in $\dot V\times \dot H^4$. Furthermore, (i) it is invariant
under the nonlinear semigroup $S(t)$. (ii) $\mathcal{E}$ is constant
on $\omega(\bov_0, \vp_0)$. (iii) $\omega(\bov_0, \vp_0)$
 consists of steady states of system \eqref{1b}--\eqref{3b}.
 \el
We note that the energy inequality obtained in Lemma \ref{A22d} not
only yields uniform higher-order energy estimates of weak solutions
(cf. \eqref{stima-rego1}), but also indicates that the asymptotic
limit points of weak solutions to problem \eqref{1b}--\eqref{3b}
actually have a special form.
 \bl
 \label{vcon}
 For $\bov_0\in \dot H$, $\varphi_0\in \dot H^2$, the weak solutions
to \eqref{1b}--\eqref{3b} have the following property
 \be
 \lim_{t\rightarrow +\infty} (\|\bov(t)\|_{\mathbf{H}^1}+ \|\mathcal{Q}(t)\|)=0.
 \label{vcon1}
 \ee
 \el
 \begin{proof}
 We recall that for any $t_1>0$, \eqref{finiA} holds. Using Lemma \ref{A22d}
 and \cite[Lemma 6.2.1]{Z04}, we conclude that $\lim_{t\rightarrow +\infty} \mathbf{A}(t)=0$,
  which together with Corollary \ref{ninf} and the Poincar\'e inequality leads to our conclusion.
 \end{proof}
 Lemma \ref{vcon} implies that for any initial data $\bov_0\in \dot H$, $\varphi_0\in \dot H^2$,
 their corresponding asymptotic limit points $(\bov_\infty, \vp_\infty)$  satisfy the following stationary
 problem (using the form \eqref{1b1}):
 \bea
 &&-\nabla P_\infty=\bov_\infty=0,\label{s1a}\\
 &&-K\Delta^2\vp_\infty+\nabla\cdot
f(\nabla \vp_\infty)=0,\label{s2a}
\\
&&\int_{\T^2}\vp_\infty dx=0.\label{s3a}
  \eea
 \bl\label{stal}
 When $n=2,3$, for any $\vp_0\in L^1$, problem
 \be
 -K\Delta^2\vp+\nabla\cdot
f(\nabla \vp)=0,\quad \int_{\T^n}\vp dx=\int_{\T^n}\vp_0dx
\label{se}
  \ee admits at least one weak
solution $\phi$, which is in fact smooth. If $\epsilon$ is properly
large, then the weak solution is unique.
 \el
 \begin{proof}
 It is easy to verify that the energy  $E(\vp)=\frac{K}{2}\|\Delta \vp\|^2+\int_{\T^2}F(\nabla
   \vp)dx$ admits at least one minimizer $\phi$ in $H^2\cap \{\vp\in L^1, \int_{\T^n}\vp dx=\int_{\T^n}\vp_0dx\}$, which is
   a weak solution to problem \eqref{se}.
    Moreover, for any weak solution $\phi$ to \eqref{se}, we
   have
    $$K\|\Delta\phi\|^2+\frac{1}{\epsilon^2}\int_{\T^n}|\nabla\phi|^4dx=\frac{1}{\epsilon^2}\|\nabla
   \phi\|^2\leq \frac{1}{2\epsilon^2}\int_{\T^n}|\nabla\phi|^4dx+ \frac{1}{2\epsilon^2}|\T^n|, $$
   which together with the Poincar\'e inequality $\|\phi\|\leq C_P (\|\nabla \phi\|+ |\int_{\T^n} \phi dx|)$
   implies that $\|\phi\|_{H^2}$ can be bounded by a constant
   depending on $|\int_{\T^n} \vp_0 dx|$, $\epsilon$, $C_P$ and
   $|\T^n|$. A bootstrap argument yields that $\phi$ is actually
   smooth and for $m\in \mathbb{N}$, $\|\phi\|_{H^m}$ can be bounded by a constant
   depending on $|\int_{\T^n} \vp_0 dx|$, $\epsilon$, $C_P$ and
   $|\T^n|$. Finally, let $\phi_1$ and $\phi_2$ be two solutions of
   \eqref{se}, using the fact
   that $\int_{\T^n} (f(\nabla \phi_1)-f(\nabla \phi_2))\cdot \nabla (\phi_1-\phi_2)dx \geq 0$, then we
   infer from the Poincar\'e inequality that
   $   K\|\Delta(\phi_1-\phi_2)\|^2\leq \frac{1}{\epsilon^2}\|\nabla
   (\phi_1-\phi_2)\|^2\leq \frac{C_P^2}{\epsilon^2}\|\Delta
   (\phi_1-\phi_2)\|^2$. As a result, if $\epsilon>C_PK^{-\frac12}$,
   then $\phi_1=\phi_2$. (We refer to \cite{LLW04}
   for a similar problem but with different boundary conditions)
 \end{proof}

\subsection{Convergence to equilibrium}

Lemma \ref{vcon} yields the convergence of velocity field $\bov$. In
what follows, we study the convergence for $\vp$. First,
$\epsilon>C_PK^{-\frac12}$, we infer from Lemma \ref{stal} that
$\omega(\bov_0, \vp_0)$ consists of a single point $(0, \vp_\infty)$
where $\vp_\infty$ is the unique solution to
\eqref{s2a}--\eqref{s3a}. However, if $\epsilon\leq
C_PK^{-\frac12}$, we lose the uniqueness of steady states.
Alternatively, we shall use the \L ojasiewicz--Simon approach.
Denote $A=-\Delta$ with $D(A)=\{\phi\in H^2, \int_{\T^n} \phi
dx=0\}$. Then $A$ is self-adjoint and positive definite. Let $H_A$
be the dual space of $H^1_*=\{\phi\in H^1,\int_{\T^n} \phi dx=0 \}$.
Then the norm on $H_A$ is given by $\|\phi\|_A^2=\int_{\T^n}\phi
A^{-1}\phi dx =\|A^{-\frac12}\phi\|^2$.

We introduce the following \L ojasiewicz--Simon type inequality:
   \bl
   \label{ls}
   Suppose $n=2,3$. Let $\psi$ be the critical point of energy
   \be
   E(\vp)=\frac{K}{2}\|\Delta \vp\|^2+\int_{\T^n}F(\nabla
   \vp)dx.\label{Evp}
   \ee
   Then, there exist constants $\beta>0$, $\theta\in(0, \frac12)$
   depending on $\psi$ such that for any $\vp\in H^3$
   with $\|\vp-\psi\|_{H^2}<\beta$ and $\int_{\T^n}\vp
   dx=\int_{\T^n}\psi dx$, there holds
   \be
   \|-K\Delta^2 \vp+\nabla \cdot f(\nabla \vp)\|_A \geq
   |E(\vp)-E(\psi)|^{1-\theta}.
   \label{LoSi}
   \ee
   \el
   \begin{proof}
   Slightly modifying the arguments in \cite{RH98,HJ99},
   we can easily prove that there exist constants $\beta_1>0$, $\theta\in(0, \frac12)$
   depending on $\psi$ such that for any $\vp\in H^3$
   with $\|\vp-\psi\|_{H^3}<\beta_1$ and $\int_{\T^n}\vp
   dx=\int_{\T^n}\psi dx$, \eqref{LoSi} holds.   Next, we slightly relax
   the smallness condition and show that \eqref{LoSi} still holds if one only
   requires that $\vp$ falls into a certain $H^2$-neighborhood of $\psi$.
  For any $\vp\in H^3$ satisfying $\int_{\T^n}\vp
   dx=\int_{\T^n}\psi dx$, using the regularity theory for elliptic
   problem, we can see that
   \be \|\vp-\psi\|_{H^3}\leq M\|\Delta^2 (\vp-
   \psi)\|_A,
   \label{55}
   \ee
where $M$ is a constant independent of $\vp$. On the other hand, if
$\|\vp-\psi\|_{H^2}\leq 1$ (which implies that $ \|\vp\|_{H^2}\leq
\|\psi\|_{H^2}+1$), then by Sobolev embedding theorem, we get
 \bea
 && \|\nabla \cdot f(\nabla \vp)-\nabla \cdot f(\nabla
\psi)\|_A\leq C_1\|\vp-\psi\|_{H^2},\non\\
 &&
 |E(\vp)-E(\psi)|^{1-\theta}\leq
 C_2\|\vp-\psi\|^{1-\theta}_{H^2},\non
 \eea
 where $C_1, C_2$ depend on $\|\psi\|_{H^2}$ and $\|\vp\|_{H^2}$ (by our assumption,
 the later one can be bounded by using only $\|\psi\|_{H^2}$). As a consequence, there
 exists a (sufficiently small) $\beta\in (0,1]$ independent of $\vp$, such that if
 $\|\vp-\psi\|_{H^2}<\beta$, then
 \be  \|\nabla \cdot f(\nabla \vp)-\nabla \cdot f(\nabla
\psi)\|_A+ |E(\vp)-E(\psi)|^{1-\theta}< \frac{\beta_1K}{2M}.
  \label{551}
  \ee
 Now for any $\vp\in H^3$ satisfying $\int_{\T^n}\vp
   dx=\int_{\T^n}\psi dx$ and  $\|\vp-\psi\|_{H^2}<\beta$, there are only two possibilities:
(i) If $\|\vp-\psi\|_{H^3}<\beta_1$, then \eqref{LoSi} holds. (ii)
If $\|\vp-\psi\|_{H^3}\geq \beta_1$,  noticing that $\psi$ satisfies
\eqref{s3a}, we deduce from (\ref{55}) and (\ref{551}) that
\begin{eqnarray}
 \| -K\Delta^2 \vp+\nabla \cdot f(\nabla \vp) \|_A
& = &\|-K\Delta^2 (\vp-\psi)+\nabla \cdot f(\nabla \vp)-\nabla \cdot
f(\nabla \psi) \|_A\nonumber\\
&\geq &K\| \Delta^2 (\vp-\psi)\|_A-\|\nabla \cdot f(\nabla
\vp)-\nabla \cdot
f(\nabla \psi) \|_A\nonumber\\
&\geq & \frac{K}{M}\|\vp-\psi\|_{H^3}-\|\nabla \cdot f(\nabla
\vp)-\nabla \cdot
f(\nabla \psi) \|_A\nonumber\\
& > & \frac{\beta_1K}{2M}> |E(\vp)-E(\psi)|^{1-\theta}.\non
\end{eqnarray}
The proof is complete.
   \end{proof}

For any initial data $(\bov_0,\vp_0)\in \dot H\times \dot H^2$, it
follows from Lemma \ref{comph} that $ \|\vp\|_{H^4}$ is uniformly
bounded for $t\geq t_2>0$. Therefore, there is an increasing
unbounded sequence $\{t_i\}_{i\in\mathbb{N}}$ and a function
$\vp_\infty\in H^4$ satisfying \eqref{s2a}--\eqref{s3a} such that
   \be
   \lim_{t_i\rightarrow +\infty} \|\vp(t_i)-\vp_\infty\|_{H^3}
   =0, \quad  \lim_{t_i\rightarrow +\infty} \mathcal{E}(t_i)=E(\vp_\infty).\label{secon}
   \ee
By \eqref{3b}, we have
   \be
   \|\vp_t\|\leq \|\bov\cdot\nabla \vp\|+\|\mathcal{Q}\|\leq
   \|\nabla \bov\|\|\vp\|_{H^2}+ \|\mathcal{Q}\|.\label{dt}
   \ee
We first exclude the trivial case, i.e., that there exists a $t_0>0$
such that $\mathcal{E}(t_0)=E(\vp_\infty)$. In this case, for all
$t\geq t_0$, we deduce from (\ref{energy}) that $\|\nabla \bov(t)\|=
\|\mathcal{Q}(t)\|= 0$. It follows from \eqref{dt} that for $t\geq
t_0$, $\|\vp_t\| =0$. Namely, $\vp$ is independent of time for all
$t\geq t_0$. Due to \eqref{secon}, we conclude that $\vp(t)\equiv
\vp_\infty$ for $t\geq t_0$. In this case, there is nothing else to
prove.

Therefore, without loss of generality, for all $t>0$, we suppose
that $\mathcal{E}(t)>E(\vp_\infty)$. For arbitrary $t>0$, we know
that $\vp\in L^2(t,t+1;H^4)\cap H^1(t, t+1; L^2)$ which implies that
$\vp\in C([t,t+1], H^2)$. Due to this continuity, by a standard
contradiction argument (see \cite{J981}), we can prove that there is
a (sufficiently large) $t_0>0$ such that for all $t\geq t_0$,
$\|\vp(t)-\vp_\infty\|_{H^2}<\beta$. Namely, for all $t\geq t_0$,
$\vp(t)$ satisfies the conditions in Lemma \ref{ls}. Apply Lemma
\ref{ls}, \eqref{energy} and the Poincar\'e inequality, we obtain
 \bea
&&-\frac{d}{dt}(\mathcal{E}(t)-E(\vp_\infty))^\theta =
-\theta(\mathcal{E}(t)-E(\vp_\infty))^{\theta-1}\frac{d}{dt}\mathcal{E}(t)
\geq \theta \frac{\frac{\mu_4}{2}\|\nabla \bov\|^2+\lambda
\left\|\mathcal{Q}\right\|^2}{\|\bov\|^{2(1-\theta)}+
\|\mathcal{Q}\|_A}\non\\
&\geq &  C(\|\nabla \bov\|+ \|\mathcal{Q}\|),
 \quad \forall\  t\geq t_0,\label{inta}
 \eea
which implies that
 $ \int_{t_0}^{+\infty} (\|\nabla \bov(\tau)\|+
 \|\mathcal{Q}(\tau)\|)d\tau<+\infty, $
 and by \eqref{dt}, $\int_{t_0}^{+\infty} \|\vp_t(\tau)\|
 d\tau<+\infty$. This easily yields the convergence of $\vp(t)$ in $L^2$ as $t\rightarrow +\infty$. Since $\vp$ is
compact in $H^3$, we infer from \eqref{secon}  that
 $\lim_{t\rightarrow +\infty}
 \|\vp(t)-\vp_\infty\|_{H^3}=0$.
By the Sobolev embedding theorem, we have
 \be K\|\Delta^2\vp(t)- \Delta^2 \vp_\infty\|
 \leq \| \mathcal{Q}(t) \|+ \|\nabla \cdot f(\nabla \vp(t))-\nabla \cdot f(\nabla \vp_\infty)\|
 \leq \|\mathcal{Q}(t)\|+C\|\vp(t)-\vp_\infty\|_{H^2},\label{kkk}
 \ee
 where $C$ depends on  $\|\vp(t)\|_{H^3}$ and $\|\vp_\infty\|_{H^3}$. As
 a consequence, we can conclude from \eqref{vcon1} and the $H^3$-convergence of $\vp$ that
  \be \lim_{t\rightarrow +\infty}
 \|\vp(t)-\vp_\infty\|_{H^4}=0.\label{conh2}
 \ee

 \subsection{Convergence rate}
  It remains to prove the convergence rate.
  This can be done in two steps: the first consists in obtaining, via the \L
ojasiewicz--Simon inequality (cf. e.g., \cite{HJ01}), the
convergence rate for the lower order terms. In the second step, we
will use the energy method to obtain the convergence rate for the
higher order terms. From Lemma \ref{ls} and \eqref{inta}, we have
 \bea
 \frac{d}{dt}(\mathcal{E}(t)-E(\vp_\infty))+ C(\mathcal{E}(t)-E(\vp_\infty))^{2(1-\theta)}\leq 0, \quad \forall\ t\geq
 t_0,\label{ly3}
 \eea
and as a consequence,
 \be \mathcal{E}(t)-E(\vp_\infty)\leq
 C(1+t)^{-\frac{1}{1-2\theta}},\quad \forall\ t\geq
 t_0.\non
 \ee
  Integrating \eqref{inta} on
$(t,+\infty)$, where $t\geq t_0$,  it follows from \eqref{dt} that
 \be \int_t^{+\infty} \|\vp_t(\tau)\| d\tau \leq C\int_t^{+\infty}(\|\nabla \bov(\tau)\|+
 \|\mathcal{Q}(\tau)\|)d\tau\leq
 C(1+t)^{-\frac{\theta}{1-2\theta}}, \label{rate1}
 \ee
which implies
 \be
    \|\vp(t)-\vp_\infty\|\leq C(1+t)^{-\frac{\theta}{1-2\theta}}, \quad t\geq t_0.\label{rate1a}
 \ee
 It follows from the basic energy law \eqref{energy} and \eqref{s3a} that
 \be
 \frac{d}{dt}y(t)+ \frac{\mu_4}{2}\|\nabla \bov\|^2+ \lambda
 \|\mathcal{Q}\|^2\leq 0,\label{dyt}
 \ee
where
 \be
 y(t)=\frac12\|\bov(t)\|^2+\frac{K}{2}\|\Delta \vp(t)-\Delta \vp_\infty\|^2+\int_{\T^2}[F(\nabla \vp(t))
 -F(\nabla \vp_\infty)+\nabla \cdot f(\nabla \vp_\infty)( \vp(t)-\vp_\infty)]
 dx.\non
 \ee
 A direct calculation yields
 \bea
 && \int_{\T^2}[F(\nabla \vp)
 -F(\nabla \vp_\infty)+\nabla \cdot f(\nabla \vp_\infty)( \vp-\vp_\infty)]
 dx\non\\
 &=& \frac{1}{4\epsilon^2}\int_{\T^2}(\d-\d_\infty)
 \cdot\left[(|\d|^2+|\d_\infty|^2+\d\cdot\d_\infty)(\d-\d_\infty)\right]dx\non\\
 &&+ \frac{1}{4\epsilon^2}\int_{\T^2}(\d-\d_\infty)
 \cdot\left[|\d_\infty|^2(\d-\d_\infty)+
 \d_\infty(\d+\d_\infty)\cdot(\d-\d_\infty)\right]dx-\frac{1}{2\epsilon^2}\|\d -\d_\infty\|^2.\non
 \eea
 Thus, we have
 \be \left|\int_{\T^2}[F(\nabla \vp)
 -F(\nabla \vp_\infty)+\nabla \cdot f(\nabla \vp_\infty)( \vp-\vp_\infty)]
 dx\right|\leq C\|\nabla
 \vp
 -\nabla \vp_\infty\|^2,\non
 \ee
 which together with the Poincar\'e inequality implies
 \be
 y(t)\geq  \frac12\|\bov\|^2+\frac{K}{4}\|\Delta \vp-\Delta
 \vp_\infty\|^2-C\|\vp-\vp_\infty\|^2.\label{y1}
 \ee
 On the other hand, it follows from \eqref{K5} and \eqref{kkk} that
 \bea
 K\|\Delta^2\vp- \Delta^2 \vp_\infty\|
 &\leq& \|\mathcal{Q}\|+C\|\vp-\vp_\infty\|_{H^2}\leq \|\mathcal{Q}\|+C\|\Delta^2 \vp-\Delta
 ^2\vp_\infty\|^\frac12\|\vp-\vp_\infty\|^\frac12\non\\
 &\leq& \|\mathcal{Q}\|+\frac{K}{2}\|\Delta^2 \vp-\Delta
 ^2\vp_\infty\|+C\|\vp-\vp_\infty\|,\label{y2a}
 \eea
 which yields
 \bea
 y(t)&\leq& \frac12\|\bov\|^2+\frac{K}{2}\|\Delta \vp-\Delta
 \vp_\infty\|^2+C\|\nabla \vp-\nabla \vp_\infty\|^2\non\\
 &\leq& C\|\nabla \bov\|^2+ C\|\Delta^2 \vp-\Delta^2
 \vp_\infty\|\|\vp-\vp_\infty\|+  C\|\Delta^2 \vp-\Delta^2
 \vp_\infty\|^\frac12\|\vp-\vp_\infty\|^\frac32\non\\
 &\leq& C\|\nabla \bov\|^2+
 C\|\mathcal{Q}\|^2+C\|\vp-\vp_\infty\|^2.\label{y2}
 \eea
 For $t\geq t_0>0$, Lemma \ref{comph} implies that
 $\mathbf{A}(t)\leq C$ that combined with \eqref{A22da} yields
 \be
 \frac{d}{dt}\mathbf{A}(t)\leq C\mathbf{A}(t).\label{y3}
 \ee
 It follows from \eqref{rate1a}, \eqref{dyt} and \eqref{y1}--\eqref{y3} that there
 exist constants $M_1, M_2>0$ such that
 \be \frac{d}{dt}[y(t)+M_1 \mathbf{A}(t)]+ M_2[y(t)+M_1
 \mathbf{A}(t)]\leq C\|\vp(t)-\vp_\infty\|^2\leq  C(1+t)^{-\frac{2\theta}{1-2\theta}}, \quad \forall \ t\geq t_0.\label{dy2}
 \ee
 By a similar argument as in \cite{WGZ1}, we conclude from
 \eqref{dy2} that
 \bea y(t)+M_1 \mathbf{A}(t)
 &\leq& [y(t_0)+M_1 \mathbf{A}(t_0)]
e^{\gamma (t_0-t)}+Ce^{-M_2 t}\int_{t_0}^t e^{M_2\tau}
(1+\tau )^{-\frac{2\theta}{1-2\theta}}d\tau\non\\
&\leq& C(1+t)^{-\frac{2\theta}{1-2\theta}},\quad \forall\ t\geq
t_0.\label{rate2}
 \eea
 Finally, from \eqref{rate1a}, \eqref{y2a} and \eqref{rate2} we
 obtain the required estimate
 \be
 \|\bov(t)\|_{\mathbf{H}^1}+\|\vp(t)-\vp_\infty\|_{H^4}\leq C (1+t)^{-\frac{\theta}{1-2\theta}},\quad \forall t\geq
t_0.\label{rate3}
 \ee

\section{Results in $3D$}
 \setcounter{equation}{0}
  \begin{lemma}\label{hA3}
 Suppose $n=3$. We have
 \be \frac{d}{dt}\mathbf{A}(t)+\frac{\mu_4}{4}\|\Delta \bov\|^2
 +\frac{\alpha\lambda K}{2}\|\Delta \mathcal{Q}\|^2\leq C_*(\mathbf{A}^3(t)+\mathbf{A}(t)), \quad \forall\ t\geq t_1>0,
 \label{A3da}
 \ee
 where $t_1>0$ is arbitrary and $\alpha>0$, $C_*>0$ are constants depending on $\|\bov_0\|$,
 $\|\varphi_0\|_{H^2}$ and $t_1$. Moreover,
 if we assume that $\varphi_0\in
 H^3$, \eqref{A3da} holds for $t\geq 0$ with $C_*$ being dependent of $\|\bov_0\|$,
 $\|\varphi_0\|_{H^3}$.
 \end{lemma}
 \begin{proof}
 Using Lemma \ref{comp} and Corollary \ref{ninf}, we modify the
calculations in Lemma
  \ref{A22d} by using the $3D$ version of embedding theorems.
  It is not difficult to see that we are still able to choose $\alpha>0$ sufficiently small in $\mathbf{A}(t)$ such
  that our conclusion holds true.
 \end{proof}
 The existence of local strong solution for arbitrary viscosity $\mu_4>0$ is a direct consequence of
 Lemma \ref{hA3}.
 \bt
 \label{gloloc3d}
Suppose $n=3$. For any $(\mathbf{v}_0, \phi_0)\in V\times H^4$,
problem \eqref{1b}--\eqref{3b} admits a unique local strong
solution.
 \et
 \begin{proof}
  Since $\vp_0\in H^4$, we have uniform estimates for $\|\vp\|_{H^3}$ and
 $\|\nabla \vp\|_{\mathbf{L}^\infty}$.  \eqref{A3da} is valid for $t\geq 0$.
 Considering the ODE problem:
 \be \frac{d}{dt}Y(t)=C_*[(Y(t))^3+Y(t)],\quad
 Y(0)=\mathbf{A}(0),\label{ODE}
 \ee
 we denote by $I=[0,T_{max})$ the interval of existence of the maximal solution $Y(t)$. We thus have $ \lim_{t\rightarrow T_{max}^-} Y(t)=+\infty. $
 It easily follows that for any $t\in I$, $0\leq
\mathbf{A}(t)\leq Y(t)$.
 Consequently, $\mathbf{A}(t)$ exists on $I$.  This and Theorem \ref{uniqueness} imply the local
existence of a unique strong solution of problem
\eqref{1b}--\eqref{3b}.
 \end{proof}

 \bp\label{sm}
Suppose $n=3$, $\bov_0\in \dot V$, $\vp_0\in H^4$. For any $R\in (0,
\infty)$, whenever
 $\|\nabla \bov_0\|^2+\|\mathcal{Q}(0)\|^2\leq R$, there is a small constant
$\varepsilon_0\in (0,1)$ depending only on $R$ and coefficients of
the system such that either (i) Problem \eqref{1b}--\eqref{3b} has a
unique global strong
 solution $(\bov,\vp)$, or (ii) there is a $T_1\in (0, +\infty)$ such that
 $\mathcal{E}(T_1)<\mathcal{E}(0)-\varepsilon_0$.
\ep
 \begin{proof}
 The proof follows from the argument in \cite{LL95} for simplified nematic liquid crystal model.
  A statement was also given for the Smectic-A system with variable density in \cite{Liu00}
 without proof. For the convenience of the readers, we sketch
 the proof here. We suppose that $(\bov, \vp)$ is a weak solution
 with initial data $(\bov_0,\vp_0)$ such that $\|\nabla \bov_0\|^2+\|\mathcal{Q}(0)\|^2\leq
 R$. Then $C_*$ in \eqref{A3da} is determined by $R$.
 Moreover, due to Corollary \ref{ninf} we have uniform estimate
 on $\|\nabla \vp(t)\|_{\mathbf{L}^\infty}$
  for all $t\geq 0$ which only depends on $R$. We fix through $R$ the constant $\alpha$ in the definition of
 $\mathbf{A}(t)$.
 Consider the ODE problem
 \eqref{ODE} with
 $Y(0)=\max\{1,\alpha\}R\geq \mathbf{A}(0)$. Let $Y(t)$ denote the unique maximal solution defined on $[0,T_{max})$.
 The time $T_{max}$ is determined by $Y(0)$ and $C_*$ in such a way
 that it
is increasing when $Y(0)$ is decreasing.
  Now we take
  $$t_0=\frac12 T_{max}(Y(0), C_*),\quad \varepsilon_0=\frac{R t_0}{2}\min\left\{\frac{\mu_4}{2},
  \lambda\right\}.$$
  If (ii) is not true, we have $ \mathcal{E}(t) \geq \mathcal{E}(0)-\varepsilon_0$ for all $t\geq0$.
From the basic energy law \eqref{energy}, we infer that
 \be
\int_{\frac{t_0}{2}}^{t_0}\mathbf{A}(t)dt\leq
 \int_0^{\infty}\mathbf{A}(t)dt \leq \kappa\varepsilon_0,\quad
 \text{with}\ \kappa=\max\{1,\alpha\}\max\{2\mu_4^{-1}, \lambda^{-1}\}.\non
 \ee
 Hence, there exists a $t_*
\in [\frac{t_0}{2}, t_0]$ such that $\mathbf{A}(t^*)\leq
\frac{2\kappa\varepsilon_0}{t_0}\leq Y(0)$. Restarting the flow
\eqref{ODE} from $t_*$, we infer from the above argument that
$\mathbf{A}(t)$ remains bounded at least on $[0,
\frac{3t_0}{2}]\subset [0,t_*+t_0]$ with the same bound as that on
$[0,t_0]$. As a consequence, an iteration argument shows that
$\mathbf{A}(t)$ is bounded for all $t\geq 0$. The proof is complete.
 \end{proof}
As an immediate consequences of the above result, we can prove (1)
eventual regularization of weak
 solutions and (2) the
well-posedness of strong solutions near the \emph{absolute}
minimizers of energy $E$ (cf. \eqref{Evp}) (cf. \cite{LL95,Liu00}).
 \bc\label{3deab}
 Suppose $n=3$.

 (1) Let $(\mathbf{v},\vp)$ be the weak solution to problem \eqref{1b}--\eqref{3b} on $[0,+\infty)$.
 Then there is some $T^*>0$ such that
 $ \mathbf{v}\in L^\infty(T^*, \infty; V)\cap L^2_{loc}(T^*,\infty; \mathbf{H}^2)$, $ \vp\in
 L^\infty(T^*,\infty; H^4)\cap L^2_{loc}(T^*,\infty; H^6)$.

(2) Let $\vp^*\in H^2$ be an absolute minimizer of
 $E(\vp)$
 in the sense that $E(\vp^*) \leq E (\vp)$ for all $\vp\in H^2$.
 For any $\bov_0\in \dot V$, $\vp_0\in H^4$ satisfying $\|\bov_0\|_{\mathbf{H}^1}\leq 1$ and  $\|\vp_0-\vp^*\|_{H^4}\leq 1$,
 there is a constant $\sigma$ which may depend
 on coefficients of the system and $\vp^*$ such that if
 $\|\bov_0\|\leq \sigma$ and $\|\vp_0-\vp^*\|_{H^2}\leq \sigma$, then problem
 \eqref{1b}--\eqref{3b} admits a unique global strong solution.
 \ec
 Next, we improve the second part of
 Corollary \ref{3deab} by proving the well-posedness of strong solutions close to \emph{local} minimizers of the energy $E$:
 \bt \label{3dlom}
 Suppose $n=3$. Let $\vp^*\in H^2$ be a local minimizer of
 $E(\vp)$
 in the sense that $E(\vp^*) \leq E (\vp)$ for all $\vp\in H^2$ satisfying $\|\vp-\vp^*\|_{H^2}<\delta$.
 For any $\bov_0\in \dot V$, $\vp_0\in H^4$ satisfying $\|\bov_0\|_{\mathbf{H}^1}\leq 1$ and  $\|\vp_0-\vp^*\|_{H^4}\leq
 1$, there exist constants $\sigma_1,\sigma_2\in (0,1]$ which may depend
 on coefficients of the system and $\vp^*$ such that if
 $\|\bov_0\|\leq \sigma_1$ and  $\|\vp_0-\vp^*\|_{H^2}\leq \sigma_2$, then problem
 \eqref{1b}--\eqref{3b} admits a unique global strong solution.
 \et
 \begin{proof}
 Without loss of generality, we assume $\delta\leq 1$. By $C_i$,
 $i=1,2,...$ we denote constants that only depend on $\vp_*$
 and on coefficients of the system. If $\|\bov_0\|_{\mathbf{H}^1}\leq 1$ and $\|\vp_0-\vp^*\|_{H^4}\leq 1$, it is not difficult to
 see that $\|\nabla \bov_0\|^2+\|\mathcal{Q}(0)\|^2\leq R$, where
 $R$ depends only on $\vp^*$. Fix this $R$, using Proposition \ref{sm}, we can also
 fix the critical constant $\varepsilon_0$ determined by
 $R$.
 It follows from Lemma \ref{low} and Lemma \ref{comp} that $\|\bov(t)\|$ and
 $\|\vp(t)\|_{H^3}$ are uniformly bounded (by a constant depending on $\vp^*$).
Since $\mathcal{E}$ is decreasing, we can see that
 \bea
 0\leq
 \mathcal{E}(0)-\mathcal{E}(t)&=&\frac12\|\bov_0\|^2-\frac12\|\bov(t)\|^2+E(\vp_0)-E(\vp(t))
 \leq \frac12\|\bov_0\|^2+E(\vp_0)-E(\vp(t))\non\\
 &\leq& \frac12\|\bov_0\|^2+C_1\|\vp(t)-\vp_0\|_{H^2}.\label{Ediff}
 \eea
  First we require $\sigma_1\leq \min\left\{\frac12
 \varepsilon_0^\frac12,1\right\}$. Let $\beta$ denote the constant depending only on $\vp^*$ provided by Lemma \ref{ls}.
 If we are able to prove
 \be
 \|\vp(t)-\vp_0\|_{H^2}< \varpi:=\min\left\{\beta,
 \frac{\varepsilon_0}{2C_1}, \delta\right\}, \quad \forall t\geq
 0,\label{jj}
 \ee
  then we can infer from \eqref{Ediff} that
 \be
 \mathcal{E}(t)\geq \mathcal{E}(0)-\varepsilon_0,\quad \forall \ t\geq
 0,\label{cod}
 \ee
 and our conclusion immediately follows from Proposition \ref{sm}.
We prove \eqref{jj} by
 a contradiction argument. Assume $\sigma_2\leq\frac{\varpi}{4}$. Let $t_*$ denote the smallest and finite time for which
 $\|\vp(t_*)-\vp^*\|_{H^2}\geq \varpi$. Without loss of generality, we
 can assume that $\mathcal{E}(t)>E(\vp^*)$ $t\in [0,t_*)$. In fact,
 if there exists $t_{**}\in (0,t_*)$ such that
 $\mathcal{E}(t_{**})=E(\vp^*)$, since $\vp^*$ is the local
 minimizer and $\|\vp(t_{**})-\vp^*\|_{H^2}<\varpi\leq \delta$, we
 can see that $\bov(t_{**})=0$ and $\vp(t_{**})=\vp^{**}$, where
 $\vp^{**}$ is also a local minimizer (possibly different from
 $\vp^*$) satisfying \eqref{s2a}--\eqref{s3a}. Due to the uniqueness
 of strong solution, the evolution starting from $t_{**}$ will be
 stationary and hence contradicting the definition of $t_*$.
 Thus, let $\mathcal{E}(t)>E(\vp^*)$ for $t\in [0,t_*)$. We observe that the conditions in Lemma \ref{ls} are fulfilled with $\vp^*$, on the interval
 $[0,t_*)$. In analogy with \eqref{inta}, we obtain
 \be
 -\frac{d}{dt}(\mathcal{E}(t)-E(\vp^*))^\theta \geq   C_2(\|\nabla
\bov\|+ \|\mathcal{Q}\|), \quad \forall\  t\in [0,t_*),\label{intb}
 \ee
 where $C_2$ depends on $\theta, \mu_4, \lambda$. Using \eqref{dt}, we have
 \bea
 \int_0^{t_*}\|\vp_t(t)\|dt &\leq&
 C_3(\mathcal{E}(0)-E(\vp^*))^\theta
 \leq C_3\left(\frac12\right)^\theta\|\bov_0\|^{2\theta}+
 C_3|E(\vp_0)-E(\vp^*)|^\theta\non\\
 &\leq& C_4 \|\bov_0\|^{2\theta}+
 C_5\|\vp_0-\vp^*\|_{H^2}^\theta.\non
 \eea
 As a result,
 \bea
 \|\vp(t_*)-\vp^*\|_{H^2}&\leq&
 \|\vp(t^*)-\vp_0\|_{H^2}+\|\vp_0-\vp^*\|_{H^2}\non\\
 &\leq&
 C\|\vp(t^*)-\vp_0\|_{H^3}^\frac23\|\vp(t^*)-\vp_0\|^\frac13+\|\vp_0-\vp^*\|_{H^2}\non\\
 &\leq& C(\|\vp(t^*)\|_{H^3}+\|\vp_0\|_{H^3})^\frac23\left( \int_0^{t_*}\|\vp_t(t)\|dt
 \right)^\frac13+\|\vp_0-\vp^*\|_{H^2}\non\\
 &\leq& C_6 \left(\|\bov_0\|^{\frac{2\theta}{3}}+
 \|\vp_0-\vp^*\|_{H^2}^\frac{\theta}{3}\right)+\|\vp_0-\vp^*\|_{H^2},\label{difff}
 \eea
Taking
 \be
 \sigma_1\leq \min\left\{\frac12 \varepsilon_0^\frac12,1,
 \left(\frac{\varpi}{4C_6}\right)^\frac{3}{2\theta}\right\}, \quad
 \sigma_2\leq\min\left\{\frac{\varpi}{4},
 \left(\frac{\varpi}{4C_6}\right)^\frac{3}{\theta}\right\},\label{ssig}
 \ee
 we easily infer from \eqref{difff} that
 $\|\vp(t_*)-\vp^*\|_{H^2}\leq \frac34\varpi<\varpi$, which leads to
 a contradiction with the definition of $t^*$. Hence, we have shown that \eqref{jj}
 holds for all $t\geq 0$. The proof is complete.
 \end{proof}
 \br
 The above proof indicates that if the (regular) initial data are properly close
 to certain local minimizer, then the global strong solution will
 remain in the neighborhood of this local minimizer for all time. The conclusion
 is also true for the case with absolute minimizer in Corollary
 \ref{3deab}. If the minimizer is isolate, then we obtain the
 stability of it.
 \er


 \bl\label{l3d}
 Suppose $n=3$. Denote
 $ \mathbf{A}_1(t)=\|\nabla \mathbf{v}(t)\|^2+\|\mathcal{Q}(t)\|^2$ and
 $\tilde{\mathbf{A}}_1(t)=\mathbf{A}_1(t)+1$.
 For any $\mu_4>0$, we have
 \bea
 &&\frac{d}{dt}\mathbf{A}_1(t)+
 \left(\frac{\mu_4}{2}-M_1\mu_4^\frac12\tilde{\mathbf{A}}_1(t)\right)\|\Delta \mathbf{v}\|^2
  +\left(\lambda K-M_2\mu_4^{-\frac12}(1+\mu_4^{-\frac52})\tilde{\mathbf{A}}_1(t)\right)\|\Delta \mathcal{Q}\|^2\non\\
 &\leq& M_3(1+\mu_4^{-3})\mathbf{A}_1(t),\label{high3d}
 \eea
 where $M_1, M_2, M_3$ are constants depending on $\|\bov_0\|$,
 $\|\vp_0\|_{H^2}$, $\mu_1,\mu_5, \lambda, K, \epsilon$, but not on $\mu_4$.
 \el
 \begin{proof}
  We note that in the following calculation only the lower-order uniform  estimates \eqref{unilow} are used.
  The possible relaxation on the viscosity $\mu_4$ enable us to avoid
  using the $\mathbf{L}^\infty$-norm of $\nabla \vp$, which was crucial in the proof of Lemma \ref{A22d}.
  In what follows, the generic constant $C$ will only depend on
 $\|\mathbf{v}_0\|$, $\|\vp_0\|_{H^2}$, $\mu_1$, $\mu_5$, $\lambda$, $K$, $\epsilon$.

We revisit the terms on the right-hand side of \eqref{dnv}.
 \bea
 \int_{\T^3} (\bov\cdot\nabla)\bov\cdot \Delta \bov
dx&\leq& \|\Delta \mathbf{v}\|\|\nabla
\mathbf{v}\|_{\mathbf{L}^3}\|\mathbf{v}\|_{\mathbf{L}^6}\leq
C\|\Delta
 \mathbf{v}\|^\frac32\|\nabla \mathbf{v}\|^\frac32\non\\
 &\leq& \mu_4^\frac12\|\nabla
 \mathbf{v}\|^\frac43\|\Delta \mathbf{v}\|^2+C\mu_4^{-\frac12}\|\nabla \mathbf{v}\|^2.\non
 \eea
 \be
 I_1+I_{2a}+I_{2b}+I_{2c} \leq \frac{\mu_1}{4}\int_Q (d_id_j\nabla_lD_{ij})^2 dx+C \|\nabla\mathbf{v}
 \|_{\mathbf{L}^3}^2\|\nabla \d\|^2_{\mathbf{L}^6}\|\d\|^2_{\mathbf{L}^\infty},\non
 \ee
 where
 \bea
  && \|\nabla
 \mathbf{v}\|_{\mathbf{L}^3}^2\|\nabla \d\|^2_{\mathbf{L}^6}\|\d\|^2_{\mathbf{L}^\infty}
 \leq C \|\Delta
 \mathbf{v}\|^\frac32\|\mathbf{v}\|^\frac12\|\nabla \d\|^2_{\mathbf{L}^6}\|\d\|^2_{\mathbf{L}^\infty}\non\\
 &\leq& C \|\Delta
 \mathbf{v}\|^\frac32\|\mathbf{v}\|^\frac12( \|\Delta ^3 \varphi\|^\frac12\|\Delta \vp\|^\frac32+\|\Delta \vp\|^2)
 (\|\Delta ^2 \varphi\|\|\nabla \vp\|+\|\nabla \vp\|^2)\non\\
 &\leq&  C \|\Delta
 \mathbf{v}\|^\frac32\|\mathbf{v}\|^\frac12(
 \|\Delta\mathcal{Q}\|^\frac12+1)(\|\mathcal{Q}\|+1)\non\\
 &\leq&  \left[\frac{\mu_4}{24}
 +\mu_4^\frac12(1+\|\mathcal{Q}\|^2)\right]\|\Delta \mathbf{v}\|^2+C\mu_4^{-\frac12}
 (\|\mathcal{Q}\|^2+\|\nabla \mathbf{v}\|^2)\|\Delta \mathcal{Q}\|^2+C\mu_4^{-1}\|\nabla
 \mathbf{v}\|^2+C\mu_4^{-\frac12}\|\mathcal{Q}\|^2.\non
 \eea
  Next,
  \bea
  I_{2d}&\leq& C\|\d\|_{\mathbf{L}^\infty}^3\|\Delta \d\|\|\nabla
  \mathbf{v}\|_{\mathbf{L}^4}^2\leq C\|\nabla \vp\|_{\mathbf{H}^2}^\frac32\|\nabla \vp\|^\frac32_{\mathbf{H}^1}\|\nabla \Delta \vp\|\|\nabla
  \mathbf{v}\|_{\mathbf{L}^4}^2\non\\
  &\leq& C(\|\nabla \Delta \vp\|^\frac52+1)\|\Delta \mathbf{v}\|^\frac32\|\nabla \mathbf{v}\|^\frac12
 \leq C (\|\mathcal{Q}\|^\frac14+1)(\|\Delta\mathcal{Q}\|^\frac12+1)\|\Delta \mathbf{v}\|^\frac32\|\nabla
  \mathbf{v}\|^\frac12\non\\
  &\leq& \left[\frac{\mu_4}{24}+\mu_4^\frac12(\|\nabla \bov\|^\frac23+\|\mathcal{Q}\|^\frac13)\right]\|\Delta \mathbf{v}\|^2
  +C(\mu_4^{-\frac32}\|\mathcal{Q}\|+\mu_4^{-3}\|\nabla \bov\|^2)\|\Delta \mathcal{Q}\|^2\non\\
  &&+C(\mu_4^{-\frac32}+\mu_4^{-3})\|\nabla \mathbf{v}\|^2.\non
 \eea
 \bea
 I_3+I_4&\leq& \|\Delta \mathbf{v}\|\|\nabla \mathbf{v}\|_{\mathbf{L}^3}\|\d\|_{\mathbf{L}^\infty}\|\nabla
 \d\|_{\mathbf{L}^6}
 \leq \|\Delta \mathbf{v}\|^\frac32 \|\nabla \mathbf{v}\|^\frac12\|\nabla \vp\|_{\mathbf{H}^2}^\frac12\|\nabla \vp\|^\frac12_{\mathbf{H}^1}
 \|\nabla \vp\|_{\mathbf{H}^2}\non\\
 &\leq& C (\|\nabla \Delta \vp\|^\frac32+1)\|\Delta \mathbf{v}\|^\frac32\|\nabla
  \mathbf{v}\|^\frac12.\non
   \eea
 It is easy to see that $I_3+I_4$ can be bounded just like $I_{2d}$, because its order is lower.
 \bea
 I_5&=&\int_{\T^3} \Delta \mathbf{v}\cdot (\mathcal{Q}\d)dx\leq \|\Delta \mathbf{v}\|\|\mathcal{Q}\|\|\nabla
 \varphi\|_{\mathbf{L}^\infty}\leq  \|\Delta \mathbf{v}\|\|\mathcal{Q}\|(\|\mathcal{Q}\|^\frac14+1)\non\\
 &\leq& \left(\frac{\mu_4}{24}+\mu_4^\frac12\|\mathcal{Q}\|^\frac12\right)\|\Delta
 \mathbf{v}\|^2+C(\mu_4^{-\frac12}+\mu_4^{-1})\|\mathcal{Q}\|^2.\non
  \eea
 For the terms $J_1,...,J_5$ on the right-hand side of \eqref{dQt}, we have
 \bea
 J_1
 &\leq& K\|\nabla \varphi\|_{\mathbf{L}^\infty}\|\Delta \mathbf{v}\|\|\Delta
 \mathcal{Q}\|+ C\|\Delta \mathcal{Q}\|\|\nabla
 \mathbf{v}\|_{\mathbf{L}^3}\|\varphi\|_{W^{2,6}}+C\|\Delta \mathcal{Q}\|\|
 \mathbf{v}\|_{\mathbf{L}^6}\|\varphi\|_{W^{3,3}}\non\\
 &\leq& C(\|\mathcal{Q}\|^\frac12+1)\|\Delta \mathbf{v}\|\|\Delta
 \mathcal{Q}\|\leq  \frac{\mu_4}{24}\|\Delta
 v\|^2+C\mu_4^{-1}(1+\|\mathcal{Q}\|)\|\Delta
 \mathcal{Q}\|^2.\non
 \eea
 \bea
 J_2+J_4&\leq& C\|\nabla \mathcal{Q}\|(\|\nabla
 \varphi\|_{\mathbf{L}^\infty}^2+1)(\|\nabla \mathbf{v}\|\|\nabla
 \varphi\|_{\mathbf{L}^\infty}+\|\mathbf{v}\|_{\mathbf{L}^6}\|\nabla \nabla \varphi\|_{\mathbf{L}^3})\non\\
 &\leq& C\|\nabla \mathcal{Q}\|\|\nabla
 \mathbf{v}\|(\|\mathcal{Q}\|^\frac34+1)\non\\
 &\leq& C(\|\Delta
 \mathcal{Q}\|^\frac12\|\mathcal{Q}\|^\frac12+\|\mathcal{Q}\|)\|\Delta
 \mathbf{v}\|^\frac12\|\mathbf{v}\|^\frac12\|\mathcal{Q}\|^\frac34 +C(\|\Delta
 \mathcal{Q}\|^\frac12\|\mathcal{Q}\|^\frac12+\|\mathcal{Q}\|)\|\nabla \mathbf{v}\|\non\\
 &\leq& \left(\frac{\lambda K}{4}+ C\mu_4^{-\frac12}\right)\|\Delta
 \mathcal{Q}\|^2+ \mu_4^\frac12\|\mathcal{Q}\|\|\Delta
 \bov\|^2+ C(1+\mu_4^{-\frac16})\|\mathcal{Q}\|^2+ C\|\nabla \mathbf{v}\|^2.\non
 \eea
 \be
 J_3+J_5\leq C\|\nabla \mathcal{Q}\|_{\mathbf{L}^3}^2(\|\nabla
 \varphi\|_{\mathbf{L}^6}^2+1)\leq
 C(\|\Delta
 \mathcal{Q}\|^\frac32\|\mathcal{Q}\|^\frac12+\|\mathcal{Q}\|^2)
 \leq \frac{\lambda K}{4}\|\Delta
 \mathcal{Q}\|^2 +C\|\mathcal{Q}\|^2.\non
 \ee
 Collecting all the estimates and using the Young inequality, we
 can obtain that
 \bea
 &&\frac{d}{dt}(\|\nabla \mathbf{v}\|^2+\|\mathcal{Q}\|^2)+ \mu_1\int_Q (d_id_j\nabla_lD_{ij})^2 dx
 +4\mu_5\int_Q (d_k\nabla_l D_{ki})^2 dx\non\\
 && +
 \left[\frac{\mu_4}{2}-C\mu_4^\frac12(1+\|\nabla
 \mathbf{v}\|^2+\|\mathcal{Q}\|^2)\right]\|\Delta \mathbf{v}\|^2\non\\
&& +\left[\lambda K-C\mu_4^{-\frac12}(1+\mu_4^{-\frac52})(1+\|\nabla
 \mathbf{v}\|^2+\|\mathcal{Q}\|^2)\right]\|\Delta \mathcal{Q}\|^2\non\\
 &\leq& C(1+\mu_4^{-3})(\|\nabla
 \mathbf{v}\|^2+\|\mathcal{Q}\|^2),\non
 \eea which yields \eqref{high3d}.
 \end{proof}

Based on Lemma \ref{l3d}, one can prove the existence and uniqueness
of global strong solutions $(\mathbf{v}, \vp)$ to our system
provided that the viscosity $\mu_4$ is properly large.

 \bt
 \label{glolarge}
Suppose $n=3$. For any $(\mathbf{v}_0, \phi_0)\in V\times H^4$, if
$\mu_4\geq
 \underline{\mu}_4(\mathbf{v}_0,\vp_0)$ (cf. \eqref{mu4b}), problem \eqref{1b}--\eqref{3b}
  admits a unique global strong solution.
 \et
 \begin{proof}
 The crucial step is to obtain a uniform
 bound of $\mathbf{A}_1(t)$. Without loss of generality, we assume that $\mu_4\geq 1$. Then we deduce from \eqref{high3d}
 that
  \be
 \frac{d}{dt}\tilde{\mathbf{A}}_1(t)+
 \left(\frac{\mu_4}{2}-M_1\mu_4^\frac12\tilde{\mathbf{A}}_1(t)\right)\|\Delta \mathbf{v}\|^2
  +\left(\lambda K-2M_2\mu_4^{-\frac12}\tilde{\mathbf{A}}_1(t)\right)
  \|\Delta \mathcal{Q}\|^2 \leq 2M_3\tilde{\mathbf{A}}_1(t).\label{high3da}
 \ee
   \eqref{int} yields that
 $ \int_t^{t+1}\tilde{\mathbf{A}}_1(\tau) d\tau\leq \int_0^{+\infty}\mathbf{A}_1(\tau) d\tau+1
 \leq
 \max\left\{2,\frac{1}{\lambda}\right\}\mathcal{E}(0)+1=:\tilde{M}$.
 If the viscosity $\mu_4$ satisfies the following relation
 \be
 \mu_4\geq \underline{\mu}_4:=\max\{1,\kappa^2\},\quad \text{with} \
 \kappa:=\max\left\{2M_1,\frac{2M_2}{\lambda K}\right\}(\tilde{\mathbf{A}}_1(0)+2M_3\tilde{M}+2\tilde{M}).\label{mu4b}
 \ee
 then applying the classical method in \cite{LL95},  we can argue as in \cite{W10} to obtain that
 \be
 \frac{\mu_4}{2}-M_1\mu_4^\frac12\tilde{\mathbf{A}}_1(t)\geq 0,
 \quad  \lambda K-2M_2\mu_4^{-\frac12}\tilde{\mathbf{A}}_1(t)\geq
 0,\quad \forall \ t\geq 0.\non
 \ee
The proof is complete.
 \end{proof}

 Finally, we study the long-time behavior of global solutions.
 \bl
 \label{AA1} Let $n=3$, the weak (or strong) solution $(\mathbf{v},\vp)$ to
 problem \eqref{1b}--\eqref{3b} has the following property:
 \be
 \lim_{t\rightarrow+\infty} (\|\nabla \bov(t)\|+\|\mathcal{Q}(t)\|)=0. \label{AAc}
 \ee
 \el
 \begin{proof}
  Since we are only concerning the behavior of $(\mathbf{v},
  \vp)$ for large time, due to the eventual regularity of weak solutions,
  we can reduce to the case of strong solutions by a finite shift of time.
  Then we can see that $\|\nabla \bov(t)\|$ and
  $\|\mathcal{Q}(t)\|$ are uniformly bounded for $t\geq 0$.
  It follows from \eqref{A3da} that
  $\frac{d}{dt}\mathbf{A}(t)\leq C$ (similarly, from \eqref{high3d}, we have $\frac{d}{dt}\mathbf{A}_1(t)\leq C$).
  Recalling that $\mathbf{A}(t),  \mathbf{A}_1(t)\in L^1(0,+\infty)$ (cf. \eqref{int}), we arrive at the conclusion.
 \end{proof}

Based on Lemma \ref{AA1}, we are able to prove the convergence to
equilibrium result in $3D$. One can check the argument for $2D$ case
in the previous section step by step. By applying corresponding
Sobolev embedding theorems in $3D$, we can see that all calculations
in Section 4.2 are valid. Hence, the details are omitted here.

 \br Since the set of equilibria can form a continuum, the global solution obtained in Corollary
\ref{3deab} or in Theorem \ref{3dlom} will converge to an
equilibrium $\vp_\infty$ which is not necessarily the original
minimizer $\vp^*$. However, we can show that
$E(\vp_\infty)=E(\vp^*)$. To see this, we recall the definition of
$\varpi$ in the proof of Theorem \ref{3dlom}. Actually we showed
that the solution $\vp(t)$ will stay in the $H^2$-neighborhood of
$\vp^*$ with radius less than $\beta$, so does $\vp_\infty$. Then,
we can apply Lemma \ref{ls} with $\psi=\vp^*$ and $\vp=\vp_\infty$
obtaining that $| E(\vp_\infty)-E(\vp^*)|^{1-\theta}\leq
\|-K\Delta^2\vp_\infty+\nabla \cdot f(\nabla \vp_\infty)\|_A=0$.
 \er


\medskip \noindent\textbf{Acknowledgments.}  A.S. was partially
supported by PRIN 2008 {\sl ``Problemi di transizione di fase e dinamiche relative''}.
  H.W. was partially supported by NSF of China 11001058 and
NSF of Shanghai 10ZR1403800. This research started when A.S. was
visiting the School of Mathematical Sciences of the Fudan University
whose hospitality is gratefully acknowledged.

\bibliographystyle{amsplain}

\begin{thebibliography}{10}
\itemsep=0pt

\bibitem{ball97} J. M. Ball,
Continuity properties and global attractors of generalized semiflows
and the Navier-Stokes equation, J. Nonlinear Sci., \textbf{7}
(1997), 475--502 . Erratum: \textbf{8} (1998), 233.

\bibitem{Chi} R. Chill,  On the \L ojasiewicz--Simon gradient inequality,  J.
Funct. Anal.,  \textbf{201}(2)  (2003),  572--601.

\bibitem{CG10} B. Climent-Ezquerra and F.
Guill\'en-Gonz\'alez, Global in time solution and time-periodicity
for a Smectic-A liquid crystal model, Commun. Pure Appl. Anal.,
\textbf{9}(6) (2010),
     1473--1493.

\bibitem{de1} P. de Gennes, J. Phys. (Paris) Colloq. \textbf{30} (Suppl. C4), 1969.

\bibitem{de2} P. de Gennes,  Viscous flows in smectic-A liquid crystals,
Phys. Fluids, \textbf{17} (1974), 1645.


\bibitem{de} P. de Gennes and J. Prost, \emph{The Physics of Liquid Crystals}, Oxford Publication,
London, 1993.



\bibitem{E} W.-N. E, Nonlinear continuum theory of Smectic-A liquid
crystals, Arch. Rational Mech. Anal., \textbf{137} (1997), 159--175.

\bibitem{EFNT} A. Eden, C. Foias, B. Nicolaenko and R. Temam, \emph{Exponential Attractors for Dissipative Evolution Equations},
Research in Applied Mathematics, \textbf{37}, John-Wiley, New York,
1994.

\bibitem{eri} J. Ericksen,  Continuum theory of nematic liquid crystals, Res.
Mechanica, \textbf{21} (1961),  381--392.

\bibitem{EMZ}
 M.~Efendiev, A.~Miranville, and S.~Zelik,
  Exponential attractors for a nonlinear reaction-diffusion system
   in $\mathbb{R}\sp 3$, C.~R.~Acad.\ Sci.\ Paris S\'er.~I Math.,
 {\bf 330}, \textbf{8} (2000), 713--718.

\bibitem{EMZ1}
 M.~Efendiev, A.~Miranville, and S.~Zelik,
  Exponential attractors and finite-dimensional reduction
 for non-autonomous dynamical systems,
 Proc.\ Roy.\ Soc.\ Edinburgh Sect.~A,
 {\bf 135}(4) (2005),  703--730.

 \bibitem{GG10} C. Gal and M. Grasselli, Asymptotic behavior of a Cahn-Hilliard-Navier-Stokes system in
 $2D$,  Ann. Inst. H. Poincar\'e  Anal. Non Lin\'{e}aire,  \textbf{27}(1) (2010),
 401--436.

 \bibitem{HJ99} A. Haraux and M.A. Jendoubi, Convergence of bounded
weak solutions of the wave equation with dissipation and analytic
nonlinearity, Calc. Var. Partial Differential Equations, {\bf 9}
(1999), 95--124.


 \bibitem{HJ01} A. Haraux and M.A. Jendoubi, Decay estimates to equilibrium for
some evolution equations with an analytic nonlinearity, Asymptot.
Anal., \textbf{26} (2001), 21--36.


\bibitem{J981} M.A. Jendoubi, A simple unified approach to
some convergence theorem of L. Simon, J. Funct. Anal., {\bf 153}
(1998), 187--202.

\bibitem{KP} M. Kleman and O. Parodi, Covariant elasticity for smectic-A, J. de Physique, \textbf{36}
(1975), 671--681.

\bibitem{KY} R.V. Kohn and X. Yan, Upper bounds on the coarsening rate for an epitaxial growth
model, Comm. Pure Appl. Math., \textbf{56}(11) (2003), 1549--1564.

\bibitem{LLW04} M.-J. Lai, C. Liu and P. Wenston, On two nonlinear
biharmonic evolution equations: existence, uniqueness and stability,
Appl. Anal., \textbf{83}(6) (2004), 541--562.


\bibitem{les} F. Leslie, Theory of flow phenomena in liquid crystals,  Advances in Liquid
Crystals, \textbf{4} (1979), 1--81.

\bibitem{LL95} F.-H. Lin and C. Liu, Nonparabolic dissipative system
modeling the flow of liquid crystals, Comm. Pure Appl. Math.,
\textbf{48}(5) (1995), 501--537.


\bibitem{Lions} P.-L. Lions, Mathematical topics in fluid mechanics, Vol. 1: Incompressible models.
Oxford Lecture Series in Mathematics and its Applications, \textbf{
3}, Oxford Science Publications, The Clarendon Press, Oxford
University Press, New York, 1996.

\bibitem{Liu00} C. Liu, The dynamic for incompressible Smectic-A liquid crystals: existence and regularity,
 Discrete Contin. Dyn. Syst., \textbf{6}(3) (2000), 591--608.

\bibitem{LJG03} B. Li and J.-G. Liu,
 Thin film epitaxy with or without slope selection,
 European J. Appl. Math., \textbf{14}(6) (2003), 713--743.

\bibitem{MN}
J.~M{\'a}lek and J.~Ne{\v{c}}as,
  A finite-dimensional attractor for three-dimensional flow of
   incompressible fluids,
 J.~Differential Equations,
 {\bf 127}, (1996),
 498--518.



\bibitem{MP}
 J.~M{\'a}lek and D.~Pra{\v{z}}{\'a}k,
  Large time behaviour via the Method of $\ell$-trajectories,
 J.~Differential Equations,
 {\bf 181}, (2002),
 243--279.



\bibitem{mane}
R.~Ma{\~n}{\'e},
 On the dimension of the compact invariant sets of certain
              nonlinear maps,
Dynamical systems and turbulence, Warwick 1980 (Coventry,
1979/1980), 230--242, Lecture Notes in Math.,
 898, Springer, Berlin-New York, 1981.


\bibitem{mpp} P. Martin, O. Parodi and P. Pershan, Unified hydrodynamic theory
for crystals, liquid crystals, and normal fluids, Phys. Rev. A,
\textbf{6} (1972), 2401.

\bibitem{MZsur}
 A. Miranville and S. Zelik,
 Attractors for dissipative partial differential equations in
  bounded and unbounded domains, Handbook of differential equations: evolutionary equations. Vol. IV,
 103--200, Handb. Differ. Equ., Elsevier/North-Holland, Amsterdam,
 2008.


\bibitem{RH98} P. Rybka and K.-H. Hoffmann, Convergence of solutions to the equation of
quasi-static approximation of viscoelasticity with capillarity, J.
Math. Anal. Appl., \textbf{226}(1) (1998), 61--81.



\bibitem{S83} L. Simon, Asymptotics for a class of nonlinear
evolution equation with applications to geometric problems, Ann. of
Math., {\bf 118} (1983), 525--571.


\bibitem{TE} R. Temam, {\em Infinite-dimensional Dynamical Systems in
Mechanics and Physics,} Appl. Math. Sci., {\bf 68}, Springer-Verlag,
New York, 1988.

\bibitem{W10} H. Wu, Long-time behavior for nonlinear hydrodynamic
system modeling the nematic liquid crystal flows,  Discrete Contin.
Dyn. Syst., \textbf{26}(1) (2010), 379--396.

\bibitem{WGZ1} H. Wu, M. Grasselli and S. Zheng, Convergence to
equilibrium for a parabolic-hyperbolic phase-field system with
Neumann boundary conditions, Math. Models Methods Appl. Sci.,
\textbf{17}(1) (2007), 1--29.

\bibitem{ZWH} L.-Y. Zhao, H. Wu and H.-Y. Huang, Convergence to equilibrium for a
phase-field model for the mixture of two viscous incompressible
fluids, Commun. Math. Sci., \textbf{7}(4) (2009), 939--962.

\bibitem{Z04} S. Zheng,  \textit{Nonlinear Evolution Equations}, Pitman Monographs
and Surveys in Pure and Applied Mathematics, {\bf 133}, CHAPMAN \&
HALL/CRC, Boca Raton, Florida, 2004.


\end{thebibliography}

\end{document}